\documentclass[11pt,reqno]{amsart}

\usepackage{amsmath,amssymb,amsthm,latexsym,soul,cite,mathrsfs}
\usepackage{mathtools}
\usepackage{enumerate}
\usepackage{lipsum}
\usepackage{ esint }
\usepackage{color,graphicx}
\usepackage[colorlinks=true,urlcolor=blue,
citecolor=red,linkcolor=blue,linktocpage,pdfpagelabels,
bookmarksnumbered,bookmarksopen]{hyperref}

\usepackage[hyperpageref]{backref}

\newtheorem{theorem}{Theorem}[section]
\newtheorem{proposition}[theorem]{Proposition}
\newtheorem{lemma}[theorem]{Lemma}
\newtheorem{corollary}[theorem]{Corollary}

\newtheorem{definition}{Definition}
\newcommand{\E}{\mathcal{E}}

\usepackage[left=2.5cm,right=2.5cm,top=2.5cm,bottom=2.5cm]{geometry}
\numberwithin{equation}{section}

\begin{document}
	
	\title[Regularity theory for problems involving general stable operators]{Regularity theory for mixed local--nonlocal problem involving general stable operators}

	\author{Pedro Fellype Pontes}
	\author{Minbo Yang}

	\address[Pedro Fellype Silva Pontes]
	{\newline\indent Zhejiang Normal University
		\newline\indent
		School of Mathematical Sciences
		\newline\indent
		Jinhua 321004 -- People's Republic of China}
	\email{\href{fellype.pontes@gmail.com}{fellype.pontes@gmail.com}}
	
	\address[Minbo Yang]
	{\newline\indent Zhejiang Normal University
		\newline\indent
		School of Mathematical Sciences
		\newline\indent
		Jinhua 321004 -- People's Republic of China}
	\email{\href{mbyang@zjnu.edu.cn}{mbyang@zjnu.edu.cn}}
	
	\begin{abstract}

In this paper, we study the regularity of solutions to a linear elliptic equation involving a mixed local--nonlocal operator of the form
$$Lu - \operatorname{div}\big(a(x)\nabla u(x)\big)= f, \quad \text{in } \Omega \subset \mathbb{R}^n,$$
where $L$ is a general stable L\'{e}vy type operator and $a(\cdot)$ is a positive H\"{o}lder continuous weight. By establishing a maximum principle and a Liouville-type result in the entire space, we are able to derive the interior regularity and the regularity up to the boundary of the solutions under suitable assumptions on $f(x)$ and $a(x)$ . 


	\end{abstract}

	\thanks{Pedro Fellype Pontes was partially supported by NSFC (W2433017), and BSH (2024-002378) and Minbo Yang is the corresponding author who was partially supported by National Natural Science Foundation of China (12471114) and Natural Science Foundation of Zhejiang Province (LZ22A010001).}
	
\maketitle

\section{Introduction}	
In this paper, we are interested in the regularity of solutions to a mixed local and nonlocal problem. More precisely, we are going to establish regularity estimates for weak solutions to the equation
	\[
\E u\coloneqq Lu - \operatorname{div}\big(a(x)\nabla u(x)\big)= f, \quad \text{in } \Omega \subset \mathbb{R}^n. \tag{P}
	\]  
where the general stable operator $L$ is the infinitesimal generator of a symmetric stable L\'{e}vy process, defined by
	\begin{equation}\label{defL}
		Lu(x)= \int_{S^{n-1}} \int_{-\infty}^{+\infty} \dfrac{u(x+\theta r)+u(x-\theta r)-2u(x)}{|r|^{1+2s}}\ dr \ d\mu(\theta),
	\end{equation}
$\mu$ is a nonnegative finite measure on the unit sphere referred to as the \textit{spectral measure} and $s \in (0,1)$ and $a(\cdot) \in C_{\text{loc}}^{0,\alpha}(\mathbb{R}^n)$, for some $\alpha \in (0,1)$, satisfying
\begin{equation}\label{conda}
	0< a^- \le a(x) \le a^+, \quad x \in \mathbb{R}^n.
\end{equation}
Throughout our analysis, we also assume the following ellipticity condition on the nonlocal operator $L$:
	\begin{equation}\label{condL}
		0<\lambda_1 \le \inf_{\nu \in S^{n-1}} \int_{S^{n-1}} \vert \nu \cdot \theta \vert  \ d\mu(\theta), \; \mbox{and} \; \int_{S^{n-1}} d\mu \le \Lambda_1 < \infty.
	\end{equation}

\subsection{Background and Motivation}

Mixed local-nonlocal differential operators have recently attracted significant interest in mathematics, both for their rich theory and real-world applications. In general, one considers an operator of the form
	$$\mathcal{L} u = L_{\rm nonlocal}\ u + L_{\rm local} \ u,$$
where $L_{\rm nonlocal}$ is an integral operator (e.g.\ a fractional Laplacian or generator of a stable L\'{e}vy process) and $L_{\rm local}$ is a standard (often second-order) elliptic operator (e.g.\ a Laplacian or weighted Laplacian). Such mixed operators naturally arise when combining short-range diffusion and long-range jumps. Among others references, Biagi et al. \cite{BDVV1} note that superposing a Brownian diffusion (classical random walk) with a L\'{e}vy-flight jump process leads to exactly a Laplacian plus fractional Laplacian operator. This modeling perspective underlies phenomena like {\it animal foraging and population dispersal} (mixing local moves and rare long-range relocations) and {\it anomalous transport} in physics. For more details we refer the reader to the seminar paper of Dipierro and Valdinoci \cite{DV}. Mathematically, these operators pose new challenges: the interplay of local regularity and nonlocal scaling breaks standard elliptic invariance, and analysts must blend techniques for PDEs and integro-differential equations.

In recent years there has been an explosion of results on these operators. Biagi, Dipierro, Valdinoci, and Vecchi \cite{BDVV2} initiated a systematic study of the elliptic operator $-\Delta+(-\Delta)^s$. In that work they proved existence, uniqueness and regularity (interior Sobolev and boundary Lipschitz) for solutions of the mixed Dirichlet problem, as well as both weak and strong maximum principles. Next, Su, Valdinoci, Wei and Zhang \cite{SVWZ2} studied the problem
	$$	-\Delta u+(-\Delta)^s u= g(x,u),$$
in bounded domain, with $g$ satisfying suitable conditions, where they proved the interior $C^{2,\alpha}$-regularity for all $s \in (0,1)$, and the $C^{2,\alpha}$-regularity up to the boundary for all $s \in (0,\frac{1}{2})$.
For the parabolic case, Chen, Kim and Song \cite{CKS} were the first to derive sharp heat-kernel estimates for $\Delta + a^\alpha\Delta^{\alpha/2}$ in bounded domains, describing the fundamental solution of a Brownian motion with an independent symmetric $\alpha$-stable jump component. 
Furthermore, a substantial body of literature addresses mixed local--nonlocal problems, each with distinct features. Notable examples include classical works such as \cite{CKSV1,CKSV2,CK,F,BI}, along with recent developments found in \cite{GK,BLS,BDVV,DfM}.


Recently, the operator $L$ defined in \eqref{defL} has attracted considerable interest within the mathematical community, primarily due to its generality and the wide range of applications it encompasses. This operator extends several classical nonlocal operators and serves as a unifying framework for analyzing various nonlocal phenomena. In particular, operators of the form $L$ are central to models of {\it nomalous diffusion} in physics, where particles undergo L\'{e}vy-type jumps instead of classical Brownian motion. In fractional {\it quantum mechanics}, the standard Laplacian in the Schrödinger equation is replaced by a stable L\'{e}vy generator, giving rise to a fractional Schrödinger equation. Moreover, in {\it mathematical finance}, such operators appear in option pricing models involving jump processes (see, e.g., \cite{BNR,MS,N,ST} and the references therein).

To illustrate the flexibility of the operator $L$, let us discuss several notable special cases:

\begin{itemize}
	\item When the spectral measure is absolute continuous, i.e., $d\mu(\theta) = \mathfrak{a}(\theta)d\theta$, then theses operators can be written as
		$$L u(x) = \int_{\mathbb{R}^n} \dfrac{u(x+y)+u(x-y)-2u(x)}{|y|^{n+2s}} \mathfrak{a}\left(\dfrac{y}{|y|}\right) \ dy,$$
	where $\mathfrak{a}\in L^1(S^{n-1})$ is a nonnegative and even function.
	
	\item When $d \mu(\theta)=cd\theta$, with $c>0$, the operator $L$ turns into a multiple of the fractional Laplacian, i.e.,
		$$L u(x) = c \int_{\mathbb{R}^n} \dfrac{u(x+y)+u(x-y)-2u(x)}{|y|^{n+2s}}  \ dy.$$
		
	\item Consider now $X_t = (X_t^1, \dots, X_t^n)$, $X_t^i$ being independent symmetric stable process in dimension $1$. In this case, the spectral measure $\mu$ is the summation of Dirac measures
		$$\mu = \sum_{i=1}^n (\delta_{e_i}+ \delta_{-e_i}),$$
	on the unit orthogonal basis $\{e_i\}$. Then, the infinitesimal generators $L$ of $X_t$ coincides with
		$$\mathcal{I}u = \sum_{i=1}^n \mathcal{I}_i u,$$
	where
		$$\mathcal{I}_i u(x) = \dfrac{c}{2} \int_{-\infty}^{+\infty} \dfrac{u(x+ye_i)+u(x-ye_i)-2u(x)}{|y|^{1+2s}}\ dy.$$
\end{itemize}

There is now a substantial body of literature concerning problems driven by the operator $L$. Notable contributions include works by Bass \cite{B}, Bass and Chen \cite{BC}, Birindelli, Du and Galise \cite{BDG}, Bogdan and Sztonyk \cite{BS}, Caffarelli and Silvestre \cite{CS}, Du and Yang \cite{DY}, Fernández-Real and Ros-Oton \cite{FR}, and Ros-Oton and Serra \cite{RS}, among many others. Of particular importance is the seminal paper \cite{RS}, in which the authors developed sharp interior and boundary regularity estimates for solutions to equations involving general symmetric stable operators.
To be more precise, consider the equation $Lu = f$ in $B_1$ with bounded right-hand side. The main interior results in \cite{RS} establish the following estimates:
	\begin{itemize}
		\item If $f \in C^\beta(B_1)$, and $u \in C^\beta(\mathbb{R}^n)$ for some $\beta>0$, then
		$$\|u\|_{C^{2s+\beta}(B_{\frac{1}{2}})} \le C \Big(\|u\|_{C^\beta(\mathbb{R}^n)} + \|f\|_{C^\beta(B_1)}\Big),$$
		whenever $2s+\beta$ is not an integer.
		
		\item If $f \in L^{\infty}(B_1)$, and  $u \in L^{\infty}(\mathbb{R}^n)$ then
			$$\|u\|_{C^{2s}(B_{\frac{1}{2}})} \le C \Big(\|u\|_{L^\infty(\mathbb{R}^n)} + \|f\|_{L^\infty(B_1)}\Big), \quad \mbox{if} \ s \neq \dfrac{1}{2},$$
			and
			$$\|u\|_{C^{2s-\varepsilon}(B_{\frac{1}{2}})} \le C \Big(\|u\|_{L^\infty(\mathbb{R}^n)} + \|f\|_{L^\infty(B_1)}\Big), \quad \mbox{if} \ s = \dfrac{1}{2},$$
		for all $\varepsilon>0$;
	\end{itemize}
To establish the regularity results mentioned above, the authors of \cite{RS} employed an strategy that hinges on a combination of heat kernel techniques, Liouville-type theorems, and a blow-up analysis. Armed with this Liouville theorem, the authors carried out a compactness and contradiction argument: they assumed, for the sake of contradiction, that the desired H\"{o}lder regularity fails. Then, by rescaling and translating the sequence of solutions (i.e., constructing a blow-up sequence), they obtained a limiting function that is a global solution to a homogeneous equation driven by a translation-invariant operator. The Liouville theorem then forces this limit to be trivial, which contradicts the nontrivial behavior in the normalization of the blow-up sequence. This contradiction ultimately establishes the desired regularity estimates.

\subsection{Main results}

Motivated by the discussions and developments presented in the previous subsection, we are now led to investigate the regularity properties of solutions to the mixed local-nonlocal linear problem $\E u = f$, where the operator $\E$ is defined in $(\mathrm{P})$ with $L$ being the infinitesimal generator of a symmetric stable L\'{e}vy process, and $a(\cdot) \in C_{\text{loc}}^{0,\alpha}(\mathbb{R}^n)$ satisfying \eqref{conda}, for some $\alpha \in (0,1)$.


Before stating our main results, we begin by defining the notion of a weak solution to our problem:

\begin{definition}
	Given $f \in L^\infty(\Omega)$, we say that $u$ is a weak solution of
	$$\E u=f, \quad \mbox{in} \ \Omega \subset \mathbb{R}^n.$$
	if
	\begin{itemize}
		\item $\vert u(x)\vert \le C(1+\vert x \vert^{2s-\delta})$ in $\mathbb{R}^n$ for some $\delta>0$;
		
		\item $\displaystyle \int_{\mathbb{R}^n} u Lv \ dx + \int_{\mathbb{R}^n}a(x)\nabla u \cdot\nabla v \ dx = \int_{\Omega} fv \ dx,$ for any $v \in C_0^\infty(\Omega)$.
	\end{itemize}
\end{definition}

The first two results in the present paper are about the interior regularity. The main idea of our method inspired by \cite{RS} is to first establish a Liouville-type result for solutions to $\E u = 0$ in $\mathbb{R}^n$, under appropriate growth assumptions (see Theorem \ref{Liouville1}, and Corollary \ref{coroLiouville} below). This result serves as a crucial rigidity tool. The first regularity result can be read as follows.

\begin{theorem}\label{reg1}
	Let $\E$ be an operator as defined in $(\mathrm{P})$, where $a(\cdot) \in C_{\mathrm{loc}}^{0,\alpha}(\mathbb{R}^n)$, for some $\alpha \in (0,1)$, satisfying \eqref{conda}, and $L$ satisfies \eqref{defL} and \eqref{condL} with $s \in (0,1)$. Assume that $u$ is any bounded weak solutions to
	$$\E u = f \ \mbox{in} \ B_1$$
	with $f \in C^\gamma(B_1)$, for some $\gamma \in (0,1)$. Suppose that one of the conditions below is satisfied:
	\begin{enumerate}[$(a)$]
		\item If $\frac{1 - \alpha}{2} < s < 1$, assume that $\gamma \in (0,1)$ is such that $\lfloor 2s+\alpha+\gamma\rfloor \le 2$;
		
		\item If $0 < s \le \frac{1 - \alpha}{2}$, suppose in addition that $\gamma \in (0,1)$ is such that $2s + \gamma + \alpha \ge 1$.
	\end{enumerate}	
	If $u \in C^\gamma(\mathbb{R}^n)$, then
	$$\|u\|_{C^{2s+\alpha+\gamma}(B_{\frac{1}{2}})} \le C \Big(\|u\|_{C^\gamma(\mathbb{R}^n)} + \|f\|_{C^\gamma(B_1)} \Big),$$
	whenever $2s+\alpha+\gamma$ is not an integer.
\end{theorem}

It is important to highlight that, due to the presence of the coefficient function $a(\cdot)$ in the definition of the operator $\E$, our regularity results must be divided into two different cases, depending on the relationship between the fractional order $s$ and the H\"{o}lder continuity exponent $\alpha$ of $a(\cdot)$. This bifurcation is essential and marks a key difference from the result obtained in \cite{RS}, where the operator was purely nonlocal.

However, this structural dependence on the interplay between $s$ and $\alpha$ is not unexpected. In fact, similar constraints appear in the recent work of Byun, Lee, and Song \cite{BLS}, where the regularity of minimizers is also derived under conditions involving these two exponents. These conditions are not merely technical: they reflect the delicate interaction between the local and nonlocal terms in the mixed operator, which must be carefully balanced to ensure that the solution inherits sufficient smoothness from both components.

Now, if we impose weaker {\it a priori} assumptions on the regularity of the right-hand side $f$ and the solution $u$, we are still able to derive meaningful estimates. In this more general setting, the regularity theory must be adapted to account for the reduced smoothness of the data, leading to a slightly weaker regularity result for solutions.

\begin{theorem}\label{reg2}
	Let $\E$ be an operator as defined in $(\mathrm{P})$, where $a(\cdot) \in C_{loc}^{0,\alpha}(\mathbb{R}^n)$, for some $\alpha \in (0,1)$, satisfying \eqref{conda}, and $L$ satisfies \eqref{defL} and \eqref{condL}, with $s \in \big[\frac{1-\alpha}{2},1\big)$. Suppose that $u$ is a solution of
	$$\E u = f, \quad \mbox{in} \ B_1,$$
	with $f \in L^\infty(B_1)$.  If $u \in L^\infty(\mathbb{R}^n)$, then the following estimates hold:
	\begin{itemize}
		\item If $s \neq \frac{1 - \alpha}{2}$ and $2s + \alpha \notin \mathbb{N}$, then
		$$\|u\|_{C^{2s+\alpha}(B_{\frac{1}{2}})} \le C \Big(\|u\|_{L^\infty(\mathbb{R}^n)} + \|f\|_{L^\infty(B_1)}\Big);$$
		
		\item If $s = \frac{1 - \alpha}{2}$, then for any $\varepsilon \in (0,1)$,
		$$\|u\|_{C^{2s+\alpha-\varepsilon}(B_{\frac{1}{2}})} \le C \Big(\|u\|_{L^\infty(\mathbb{R}^n)} + \|f\|_{L^\infty(B_1)}\Big).$$
	\end{itemize}
\end{theorem}

We divide our previous regularity result into two separate cases, each corresponding to different regimes of the parameters involved. In both cases, we establish H\"{o}lder continuity of the solutions under boundedness assumptions. However, unlike \cite{RS}, we were not able to treat the delicate case when $2s + \alpha < 1$. The lack of sufficient regularity in this range is due to the technical obstacles that our methods do not work well, and thus this scenario remains an open and interesting question for future research.

Moving forward, in order to prove regularity up to the boundary, we employ a maximum principle version adapted to our case, as follows.

\begin{theorem}\label{maxprin}
	Let $\Omega$ be a bounded domain in $\mathbb{R}^n$, $L$ be any operator of the form \eqref{defL},\eqref{condL}, and $a(\cdot) \in C_{\mathrm{loc}}^{0,\alpha}(\mathbb{R}^n)$, for some $\alpha \in (0,1)$, satisfying \eqref{conda}. Assume that $u\in H^1(\mathbb{R}^n)$ satisfies $\mathcal{E}u \ge 0$ in $\Omega$ weakly. If $u \ge 0$ in $\mathbb{R}^n\setminus\Omega$, then $u \ge 0$ almost everywhere in $\Omega$.
\end{theorem}

We now establish a up to the boundary regularity result that, in contrast to the previous Theorem \ref{reg1}, does not rely on any \textit{a priori} regularity assumption on the solution $u$. In this setting, we restrict to the case $a \equiv 1$, and the precise statement reads as follows.

\begin{theorem}\label{boundary_reg}
	Let $\Omega$ be a $C^{1,1}$ domain in $\mathbb{R}^n$, $L$ be any operator of the form \eqref{defL},\eqref{condL} with $s \in (0,1)$, and $a(x)\equiv 1$. Assume that $u$ is a weak solution of
		$$\left\{
		\begin{array}{rcccl}
			\mathcal{E} u & =& f & \text{in} & \Omega; \\
			u & = & 0 & \text{in} & \mathbb{R}^n \setminus \Omega,
		\end{array}
		\right.$$
	for a given $f \in L^\infty(\Omega)$. If $u \in L^\infty(\Omega)$, then $u \in C^{1,\gamma}(\overline{\Omega})$ for any $\gamma \in (0,\min\{1,2-2s\})$.
\end{theorem}

In fact, Theorem \ref{boundary_reg} can be regarded as a direct consequence of the $W^{2,p}$-regularity established in Theorem \ref{W2p-reg} in Section 5. Moreover, it extends Theorem \ref{reg2} (in the case $a \equiv 1$), bridging precisely the gap left open when $2s + \alpha < 1$.

We emphasize that the weight $a(\cdot)$ in the local part of the operator is not a mere technicality but also has a clear application meaning. As discussed before, in ecological models (see \cite{DV,DLV}), $u$ describes the density of a population within a habitat, and the mixed operator reflects the combination of Brownian and L\'{e}vy-type movements. The function $a(x)$ captures environmental heterogeneity: it is larger in open or favorable regions, enhancing local diffusion, and smaller in fragmented or adverse areas, where nonlocal dispersal becomes more significant.


To the best of our knowledge, this is the first result in the mixed local--nonlocal operator involving the general stable operator, even when the weigh $a(\cdot)$ is constant.

The structure of the paper is organized as follows. In Section \ref{Sec2}, we prove the maximum principle Theorem \ref{maxprin} and establish, by using tools from the theory of the Heat Kernel, a Liouville-type theorem in the entire space $\mathbb{R}^n$. Sections \ref{Sec3} and \ref{Sec4} are dedicated to the proofs of Theorems \ref{reg1} and \ref{reg2}, respectively. In both cases, the approach consists of first establishing regularity estimates for smooth test functions. These preliminary results are then extended to more general solutions via a standard approximation procedure. Finally, Section \ref{Sec5} is devoted to prove the up to the boundary regularity, namely Theorem \ref{boundary_reg}.

\section{A Maximum principle and Liouville-type theorem}\label{Sec2}

In this section, we establish Theorem \ref{maxprin} and prove a Liouville-type theorem in the whole space by means of an argument based on the heat kernel. We begin with the maximum principle.

To set the framework, let us clarify what we mean when we say that a function $u$ satisfies $\mathcal{E}u \ge 0$ weakly in $\Omega$. Specifically, this condition is understood in the sense that
	\begin{equation}\label{def>0}
		\int_{\mathbb{R}^n} u Lv \ dx + \int_{\mathbb{R}^n}a(x)\nabla u \cdot\nabla v \ dx \ge 0, \quad \mbox{for any}\ v \in C_0^\infty(\Omega).
	\end{equation}
By standard density and extension arguments, this inequality continues to hold for any function $v \in H^1(\mathbb{R}^n)$.
	
\begin{proof}[Proof of Theorem \ref{maxprin}]
	First, note that standard computation yields
		$$\int_{\mathbb{R}^n} u(x) Lv(x)\ dx = \dfrac{1}{2} \int_{\mathbb{R}^n}\int_{S^{n-1}}\int_{-\infty}^\infty \dfrac{\big(u(x)-u(x+\theta r)\big) \cdot \big(v(x)-v(x+\theta r)\big)}{|r|^{1+2s}} \ dr \ d\mu \ dx.$$
	We now argue by contradiction. Suppose that there exists a measurable set $E \subset \Omega$ of positive measure such that $u < 0$ in $E$. Define the positive and negative parts of $u$ as
		$$u_+ \coloneqq \max\{u,0\} \quad \mbox{and} \quad u_- \coloneqq \max\{-u,0\},$$
	so that $u = u_+ - u_-$. It is well known that $u_- \in H^1(\mathbb{R}^n)$, which makes it a valid function in \eqref{def>0}. Substituting $v = u_-$ into \eqref{def>0}, we obtain
		\begin{eqnarray}\label{cont>0}
			0&\le& \dfrac{1}{2} \int_{\Omega}\int_{S^{n-1}}\int_{-\infty}^\infty \dfrac{\big(u(x)-u(x+\theta r)\big) \cdot \big(u_-(x)-u_-(x+\theta r)\big)}{|r|^{1+2s}} dr d\mu dx \nonumber\\
			&& \hspace{10cm}+ \int_{\Omega}a(x)\nabla u \cdot\nabla u_- dx \nonumber \\
			&=&\dfrac{1}{2} \int_{\Omega}\int_{S^{n-1}}\int_{-\infty}^\infty \dfrac{\big(u(x)-u(x+\theta r)\big) \cdot \big(u_-(x)-u_-(x+\theta r)\big)}{|r|^{1+2s}} dr d\mu dx - \int_{\Omega} a(x)|\nabla u_-|^2 dx \nonumber \\
			&\le& \dfrac{1}{2} \int_{\Omega}\int_{S^{n-1}}\int_{-\infty}^\infty \dfrac{\big(u(x)-u(x+\theta r)\big) \cdot \big(u_-(x)-u_-(x+\theta r)\big)}{|r|^{1+2s}} dr d\mu dx.
		\end{eqnarray}
	On the other hand, using the decomposition $u = u_+ - u_-$, we find
		\begin{eqnarray*}
			&&\dfrac{1}{2} \int_{\Omega}\int_{S^{n-1}}\int_{-\infty}^\infty \dfrac{\big(u(x)-u(x+\theta r)\big) \cdot \big(u_-(x)-u_-(x+\theta r)\big)}{|r|^{1+2s}} dr d\mu dx \\
			&=& \dfrac{1}{2} \int_{\Omega}\int_{S^{n-1}}\int_{-\infty}^\infty \dfrac{\big(u_+(x)-u_+(x+\theta r)\big) \cdot \big(u_-(x)-u_-(x+\theta r)\big)}{|r|^{1+2s}} dr d\mu dx\\
			&& - \dfrac{1}{2} \int_{\Omega}\int_{S^{n-1}}\int_{-\infty}^\infty \dfrac{\big|u_-(x)-u_-(x+\theta r)\big|^2}{|r|^{1+2s}} dr d\mu dx \\
			&\le&\dfrac{1}{2} \int_{\Omega}\int_{S^{n-1}}\int_{-\infty}^\infty \dfrac{\big(u_+(x)-u_+(x+\theta r)\big) \cdot \big(u_-(x)-u_-(x+\theta r)\big)}{|r|^{1+2s}} dr d\mu dx.
		\end{eqnarray*}
	Observe that we always have $\big(u_+(x)-u_+(x+\theta r)\big) \cdot \big(u_-(x)-u_-(x+\theta r)\big) \le 0$, since $u_+$ and $u_-$ have disjoint supports. Consequently, the whole expression is less than or equal to zero. This contradicts inequality \eqref{cont>0}, which required the left-hand side to be non-negative. Therefore, our assumption was false, and we conclude that $u\ge0$ almost everywhere in $\Omega$, as desired.
\end{proof}

Now, our main goal is to prove the Liouville-type theorem. To begin with, we point out that the heat kernel associated with the operator $\E$ is defined via the Fourier transform as
	$$H(t,\cdot) = \mathcal{F}\big(\exp(-A(\xi))t\big),$$
where $A(\xi)$ denotes the Fourier symbol of $\E$. This symbol can be expressed as the sum of the Fourier symbols corresponding to the operators $L$ and $-\mbox{div}(a(x)\nabla(\cdot))$. As discussed in \cite{RS} the symbol of $L$ is given by
	$$A_L(\xi) = c \int_{S^{n-1}} \vert \xi \cdot \theta \vert^{2s} \ d\mu(\theta).$$
On the other hand, due to the presence of the variable coefficient $a(x)$, the symbol $A_d(\xi)$ corresponding to $\mbox{div}(a(x)\nabla u)$ cannot be written explicitly, because it depends on the Fourier transform of $a$. However, by the condition \eqref{conda}, there exist $\lambda_2, \Lambda_2>0$ depending on the constants of \eqref{conda} such that the bounds
	$$\lambda_2 |\xi|^2 \le A_d(\xi) \le \Lambda_2 |\xi|^2,$$
holds. Combining both contributions, we deduce the following two-sided estimate for $A(\xi)$:
	\begin{equation}\label{FSbounds}
		\lambda \min \{\vert \xi\vert^{2s}, \vert \xi\vert^{2}\} \le A(\xi) \le \Lambda \max \{\vert \xi\vert^{2s}, \vert \xi\vert^{2}\},
	\end{equation}
for some constants $\lambda,\Lambda>0$.

\begin{proposition}\label{heatk}
	Let $\E$ be any operator defined as in $(P)$ with $L$ satisfying \eqref{defL} and \eqref{condL}, $s \in (0,1)$ and $a(\cdot) \in C_{loc}^{0,\alpha}(\mathbb{R}^n)$ satisfying \eqref{conda}. Then, there exists a constant $C=C(n,s,\alpha,\delta,\lambda,\Lambda)>0$ such that
		\begin{enumerate}[$(a)$]
			\item For all $\delta>0$ we have $\displaystyle \int_{\mathbb{R}^n} \big(1+\vert x \vert^{2s-\delta}\big)H(1,x) \ dx \le C$;
			
			\item Moreover, $[H(1,x)]_{C^{0,1}(\mathbb{R}^n)} \le C$.
		\end{enumerate}
\end{proposition}
\begin{proof}
	$(a)$ We begin the proof by claiming that the function $\varphi(x) = (1 + |x|^2)^{s - \delta}$ satisfies
		$$|\E \varphi(x)| \le C,$$
	for all $x \in \mathbb{R}^n$, where $C > 0$ is a constant depending only on $n$, $s$, $\alpha$, $\lambda$, and $\Lambda$. To prove this, consider the family of rescaled functions
		$$\varphi_\rho(x) = \rho^{-2s + 2\delta} \varphi(\rho x), \quad \rho \ge 1.$$
	These functions satisfy
		$$
		\begin{cases}
			\varphi_\rho(x) = (\rho^{-2} + |x|^2)^{s - \delta}, \\
			|\E \varphi_\rho(x)| \le C \quad \text{for } x \in B_2 \setminus B_1,
		\end{cases}
		$$
	with a constant $C$ independent of $\rho$. Indeed, from \cite{RS}, we know
		$$|L \varphi_\rho| \le C \quad \text{in } B_2 \setminus B_1.$$
	To estimate the second-order term, a straightforward computation shows
		$$\nabla \varphi_\rho(x) = 2(s - \delta)(\rho^{-2} + |x|^2)^{s - \delta - 1}x,$$
	and the Laplacian becomes
		$$\Delta \varphi_\rho(x) = 4(s - \delta)(s - \delta - 1)(\rho^{-2} + |x|^2)^{s - \delta - 2}|x|^2 + 2n(s - \delta)(\rho^{-2} + |x|^2)^{s - \delta - 1}.$$
	Both terms are uniformly bounded in the annulus $B_2 \setminus B_1$ for all $\rho \ge 1$, hence
		$$|\mbox{div}(a(x)\nabla \varphi_\rho)| \le C.$$
	Combining these bounds, we conclude that
	$$|\E \varphi_\rho| \le C \quad \text{in } B_2 \setminus B_1.$$
	Scaling back yields
	$$|\E \varphi(x)| \le C \rho^{2 - 2\delta} \quad \text{in } B_{2\rho} \setminus B_\rho,$$
	for any $\rho \ge 1$. This proves that $\E \varphi$ remains uniformly bounded in whole $\mathbb{R}^n$.
	
	Now, since $H(0,x)=\overline{\delta}(x)$, with $\overline{\delta}$ denoting the Dirac's delta, we have by sifting property that
		$$\int_{\mathbb{R}^n}\varphi(x)H(0,x)dx = \int_{\mathbb{R}^n}\varphi(x)\overline{\delta}(x)dx = \varphi(0) = 1.$$
	Thus, by mean value theorem, we get
		\begin{eqnarray*}
			\int_{\mathbb{R}^n}\varphi(x)H(1,x)dx -1 &=& \int_{\mathbb{R}^n}\varphi(x)\Big(H(1,x)-H(0,x)\Big)dx \\
			&=& \int_{0}^{1}\int_{\mathbb{R}^n} \varphi(x)H_t(t,x)\ dx\ dt \\
			&=& \int_{0}^{1} \int_{\mathbb{R}^n} \E \varphi(x) H(t,x)\ dx \ dt,
		\end{eqnarray*}
	provide that $H(t,x)$ is the heat kernel of $\E$. Therefore, it follows
		\begin{eqnarray*}
			\int_{\mathbb{R}^n}\varphi(x)H(1,x)dx &\le& 1 + \int_{0}^{1}\int_{\mathbb{R}^n} \vert \E\varphi\vert \cdot H(t,x) \ dx \ dt \\
			&\le& 1+ C \int_{0}^{1}\int_{\mathbb{R}^n}H(t,x)\ dx \ dt \\
			&\le& C,
		\end{eqnarray*}
	and the item $(a)$ follows.
	
	\noindent $(b)$ By condition \eqref{conda} we know that \eqref{FSbounds} holds. Thus, it allow us to see that the Fourier transform of $H(1,x)$ is rapidly decreasing and then the gradient of $H(1,x)$ is bounded, and therefore the result follows.
\end{proof}

Now we are able to state and prove the Liouville-type theorem in the entire space, as follows.

\begin{theorem}\label{Liouville1}
	Let $\E$ be any operator defined as in $(P)$, with $L$ satisfying \eqref{defL} and \eqref{condL}, $s \in (0,1)$ and $a(\cdot) \in C_{loc}^{0,\alpha}(\mathbb{R}^n)$ satisfying \eqref{conda}. Let also $u$ be any weak solution of
		$$\E u = 0, \quad \mbox{in} \ \mathbb{R}^n,$$
	satisfying the following growth condition
		$$\Vert u \Vert_{L^\infty(B_R)} \le C R^\gamma \quad \mbox{for any} \ R\ge1 \ \mbox{and some} \ \gamma<2s.$$
	Then, $u$ is a polynomial of degree at most $\lfloor \gamma \rfloor$.
\end{theorem}
\begin{proof}
	Consider $\rho\ge1$ and define $v(x) \coloneqq\rho^\gamma u(\rho x)$. Since $\E(u)=0$, then $\E (v) =0$ in the whole $\mathbb{R}^n$. Moreover, by the growth hypothesis on $u$ we have
		\begin{equation}\label{growthv}
			\|v\|_{L^\infty(B_R)} \le C R^\gamma.
		\end{equation}	
	We claim that
		\begin{equation}\label{claim1}
			v \equiv H(1,\cdot) * v,
		\end{equation}
	where \( H \) denotes the heat kernel associated with the operator \( \mathcal{E} \), and \( * \) stands for convolution. To justify this identity, observe that \( v \) is a weak solution to \( \mathcal{E}(v) = 0 \) in \( \mathbb{R}^n \). Therefore, for every test function \( \eta \in C_0^\infty\big((0,1)\times \mathbb{R}^n\big) \), we have
		\begin{equation}\label{*}
			\int_{\mathbb{R}^n} v(x-z)\, \mathcal{E} \eta(t,x)\, dx = 0,
		\end{equation}
	for all \( (t,x) \in (0,1) \times \mathbb{R}^n \) and fixed \( z \in \mathbb{R}^n \).
	
	Now, define
		$$\omega(t,x) := (H(t,\cdot) * v)(x).$$
	Using the growth estimate for \( v \) (cf. \eqref{growthv}) and Proposition \ref{heatk}(a), it follows that \( \omega \) is a weak solution of the parabolic equation
		$$\partial_t \omega = \mathcal{E} \omega \quad \text{in } (0,+\infty) \times \mathbb{R}^n.$$
	
	By integrating by parts in time, we find
	\begin{align*}
		-\int_{0}^{1} \int_{\mathbb{R}^n} \omega\, \partial_t \eta\, dx\, dt
		&= \int_{0}^{1}\int_{\mathbb{R}^n} \partial_t \omega\, \eta\, dx\, dt - \int_{\mathbb{R}^n} \omega(1,x)\eta(1,x)\, dx + \int_{\mathbb{R}^n} \omega(0,x)\eta(0,x)\, dx \\
		&= \int_{0}^{1}\int_{\mathbb{R}^n} \omega\, \mathcal{E} \eta\, dx\, dt,
	\end{align*}
	where in the last step we used that $\omega$ is a weak solution.
	
	From the definition of $ \omega $, it follows that
		$$-\int_{0}^{1} \int_{\mathbb{R}^n} \omega\, \partial_t \eta\, dx\, dt
	= \int_{0}^{1} \int_{\mathbb{R}^n} \int_{\mathbb{R}^n} H(t,z)\, v(x-z)\, \mathcal{E} \eta(t,x)\, dz\, dx\, dt.$$	
	With the aim to use the Fubini’s theorem to change the order of integration, we need to verify the integrability:
		$$\int_{0}^{1} \int_{\mathbb{R}^n} \int_{\mathbb{R}^n} H(t,z)\, |v(x-z)|\, |\mathcal{E} \eta(t,x)|\, dx\, dz\, dt < \infty.$$
	
	Note that, due to the growth control \eqref{growthv}, we have
		$$\int_{\mathbb{R}^n} |v(x-z)|\, |\mathcal{E} \eta(t,x)|\, dx \le C (1 + |z|)^\gamma,$$
	for some constant \( C > 0 \). Therefore,
		$$\int_{0}^{1} \int_{\mathbb{R}^n} H(t,z) \int_{\mathbb{R}^n} |v(x-z)|\, |\mathcal{E} \eta(t,x)|\, dx\, dz\, dt
	\le C \int_{0}^{1} \int_{\mathbb{R}^n} H(t,z)\, (1 + |z|)^\gamma\, dz\, dt.$$	
	To estimate the last integral, we split the domain of integration:
	
	\begin{itemize}
		\item For $|z| < 1$, we estimate
			$$\int_{[|z| < 1]} H(t,z)\, (1 + |z|)^\gamma\, dz \le 2^\gamma \int_{\mathbb{R}^n} H(t,z)\, dz = 2^\gamma.$$
		
		\item For $|z| > M $ (for some large $ M > 0 $), the Gaussian upper bound on the heat kernel gives
		$$H(t,z) \le C t^{-\frac{n}{2s}} e^{-\frac{|z|^{2s}}{t}} \le C |z|^{-(n + 2s)},$$
		and hence
			$$\int_{[|z| > M]} H(t,z)\, (1 + |z|)^\gamma\, dz \le C \int_{[|z| > M]} |z|^{\gamma - (n + 2s)}\, dz.$$
		Passing to spherical coordinates, we obtain
			$$\int_{[|z| > M]} |z|^{\gamma - (n + 2s)}\, dz
		= C \int_{M}^{\infty} r^{\gamma - 2s - 1}\, dr < \infty,$$
		since $ \gamma < 2s$.
		
		\item Finally, on the annulus \( 1 \le |z| \le M \), the integrand is continuous and bounded, so
		$$\int_{[1 \le |z| \le M]} H(t,z)\, (1 + |z|)^\gamma\, dz \le C.$$
	\end{itemize}	
	By combining these estimates, we conclude that
		$$\int_{0}^{1} \int_{\mathbb{R}^n} H(t,z)\, (1 + |z|)^\gamma\, dz\, dt \le C,$$
	as required. Therefore, by Fubini’s theorem,
		$$-\int_{0}^{1} \int_{\mathbb{R}^n} \omega\, \partial_t \eta\, dx\, dt
	= \int_{0}^{1} \int_{\mathbb{R}^n} \int_{\mathbb{R}^n} H(t,z)\, v(x-z)\, \mathcal{E} \eta(t,x)\, dx\, dz\, dt \stackrel{\eqref{*}}{=} 0.$$	
	By the arbitrariness of $\eta \in C_0^\infty((0,1)\times \mathbb{R}^n)$, we get
		$$\omega(1,x) = \omega(0,x),$$
	which establishes the desired identity \eqref{claim1}.
	
	Next, we claim that
		\begin{equation}\label{Holdv}
			[v]_{C^\tau(B_1)} \le C,
		\end{equation}
	for some $\tau>0$ and $C>0$ depending only on $n,\lambda,\Lambda,\gamma$.
	
	Indeed, fix $x,x'\in B_1$, with $x\neq x'$. Using identity \eqref{claim1}, we write for any $K>0$:
		\begin{eqnarray}\label{I12}
			|v(x)-v(x')| &=& \big|\big(H(1,\cdot)*v\big)(x) - \big(H(1,\cdot)*v\big)(x')\big| \nonumber\\
			&=& \left|\int_{\mathbb{R}^n} \Big(H(1,x-y)-H(1,x'-y)\Big)v(y) \ dy\right| \nonumber \\
			&\le& \left|\int_{[|y|\le K]} \Big(H(1,x-y)-H(1,x'-y)\Big)v(y) \ dy\right| \nonumber \\
			&& + 2\sup_{x \in B_1} \left|\int_{[|y|\ge K]} H(1,x-y)v(y) \ dy\right| \nonumber\\
			&\coloneqq& I_1 + I_2.
		\end{eqnarray}
	We now estimate the terms $I_1$ and $I_2$. Using Proposition \ref{heatk} $(b)$ and \eqref{growthv}
		$$|I_1| \le \|v\|_{L^\infty(B_K)} \int_{[|y|\le K]} |x-x'| \ dy \le CK^{\gamma+n}|x-x'|.$$
	While, by using Proposition \ref{heatk} $(a)$ with $\delta>0$ satisfying $2\delta=2s-\gamma$, we obtain
		\begin{eqnarray*}
			\int_{[|y|\ge K]} H(1,x-y) \ |v(y)|\ dy &\le& \int_{[|y|\ge K]} C \int_{[|y|\ge K]} H(1,x-y) \dfrac{|y|^\gamma}{(1+|y|)^{\gamma+\delta}} \\
			&\le& CK^{-\delta} \int_{\mathbb{R}^n} H(1,x-y) (1+|y|)^{2s-\delta} \ dy \\
			&\le& CK^{-\delta},
		\end{eqnarray*}
	leading to $I_2 \le CK^{-\delta}$.	Thus, by combination of $I_1$ and $I_2$, \eqref{I12} turns into
		$$|v(x)-v(x')| \le C K^{n+\gamma}|x-x'| + C K^{-\delta}.$$
	Since $K>0$ is arbitrary, we can choose
		$$K=|x-x'|^{-\frac{\tau}{\delta}}, \quad \mbox{with} \ 1-\dfrac{(n+\gamma)\tau}{\delta} = \tau.$$
	Substituting this choice of $K$, we get
		$$|v(x)-v(x')| \le C|x-x'|^{-\frac{\tau(n+\gamma)}{\delta}}|x-x'| + C|x-x'|^\tau \le C|x-x'|^\tau,$$
	for some $\tau>0$. This proves the H\"{o}lder continuity estimate \eqref{Holdv}.
	
	From the definition of $v$, the H\"{o}lder estimate \eqref{Holdv} implies that
		$$[u]_{C^\tau(B_\rho)} \le C \rho^{\gamma-\tau}, \quad \mbox{for any} \ \rho \ge1.$$
	Consider now the incremental quotient
		$$u_h^\tau \coloneqq \dfrac{u(\cdot + h)-u}{|h|^\tau},$$
	which satisfies the estimate
		$$\|u_h^\tau\|_{L^\infty(B_R)} \le C R^{\gamma -\tau}.$$
	Therefore, by applying the previous argument with $v$ replaced by $u_h^\tau$ and $\gamma$ replaced by $\gamma-\tau$, we deduce that
		$$[u_h^\tau]_{C^\tau(B_R)} \le CR^{\gamma-2\tau}.$$
	This yields
		$$[u]_{C^{2\tau}(B_R)} \le CR^{\gamma-2\tau}.$$
	By iterating this process, we obtain by induction
		$$[u]_{C^{N\tau}(B_R)} \le CR^{\gamma-N\tau}.$$
	Now choose the smallest integer $N$ such that $\gamma-N\tau<0$. Then, letting $R\to\infty$ we conclude that
		$$[u]_{C^{N\tau}(\mathbb{R}^n)} = 0.$$
	This implies that $u$ must be a polynomial of degree at most $\lfloor \gamma\rfloor$.
\end{proof}

\begin{corollary}\label{coroLiouville}
	Let $\E$ be any operator defined as in $(P)$, with $L$ satisfying \eqref{defL} and \eqref{condL}, $s \in (0,1)$, and $a(\cdot) \in C_{loc}^{0,\alpha}(\mathbb{R}^n)$ satisfying \eqref{conda}. Consider $u$ any function satisfying, in the weak sense
		$$\E\Big(u(\cdot+h)-u(\cdot)\Big)=0, \quad \mbox{in} \ \mathbb{R}^n \; \; \mbox{for any} \ h \in \mathbb{R}^n.$$
	Suppose that $u$ satisfies
		$$[u]_{C^\tau(B_R)} \le C R^\gamma, \quad \mbox{for any} \ R\ge1,$$
	for some $\gamma<2s$. Then, $u$ is a polynomial of degree at most $\lfloor\gamma+\tau\rfloor$.
\end{corollary}
\begin{proof}
	To conclude, we just need apply Theorem \ref{Liouville1} to the difference function
		$$v(x)\coloneqq u(x+h)-u(x),$$
	and deduce that $v$ is a polynomial. Since this argument is valid for every $h \in \mathbb{R}^n$, it follows that $u$ must be a polynomial as well. Additionally, the prescribed growth condition ensures that the degree of $u$ is at most $\lfloor \gamma + \tau\rfloor$.
\end{proof}

\section{Interior regularity (part one)}\label{Sec3}

In this section we will prove the first interior regularity, namely Theorem \ref{reg1}. The idea is use a compactness argument and the Liouville-type theorem established in the previous section. We will start wi the following result:

\begin{lemma}\label{lemcontr}
	Consider $\{\E_k\}_{k\in\mathbb{N}}$ a sequence of operators defined by
		$$\E_k v = L_k v - \mbox{div}(a_k(x)\nabla v),$$
	with $\{a_k\} \in C_{loc}^{0,\alpha}(\mathbb{R}^n)$, and $L_k$ in the form \eqref{defL} satisfying \eqref{condL}. Let $\{u_k\}_{k\in \mathbb{N}}$ and $\{f_k\}_{k\in \mathbb{N}}$ be sequences satisfying
		$$\E_k u_k = f_k, \quad \mbox{in} \ \Omega\subset \mathbb{R}^n,$$
	in the weak sense. Assume that $L_k$ have spectral measures $\mu_k$ converging to a spectral measure $\mu$. Let $L$ be the operator associated to $\mu$, and suppose that
		\begin{itemize}
			\item the sequence $\{u_k\}$ is bounded in $H^1(\Omega)$, while $\{a_k\}$ is uniformly bounded in $C_{loc}^{0,\alpha}(\mathbb{R}^n)$;
			
			\item $u_k \to u$ uniformly in compact sets of $\mathbb{R}^n$;
			
			\item $f_k \to f$ uniformly in $\Omega$;
			
			\item $|u_k(x)|\le C (1+|x|^{2s-\varepsilon})$ for some $\varepsilon>0$ and any $x \in \mathbb{R}^n$;
		\end{itemize}
	Then, $u$ satisfies
		$$\E u = f \quad \mbox{in} \ \Omega,$$
	in the weak sense, where $\E v = Lv - \mbox{div}(a(x)\nabla v)$, for some $a \in C_{loc}^{0,\alpha_1}(\Omega)$ with $\alpha_1 \in (0,\alpha)$.
\end{lemma}

\begin{proof}
	Since $\mathcal{E}_k u_k = f_k$, we have, for any $\eta \in C_0^\infty(\Omega)$,
		$$\int_{\mathbb{R}^n} u_k \, L_k \eta \, dx - \int_{\Omega} a_k(x) \nabla u_k \cdot \nabla \eta \, dx = \int_{\Omega} f_k \eta \, dx.$$
	On the other hand, using the inequality
		$$|\eta(x+y) + \eta(x-y) - 2\eta(x)| \le C \min\{1, |y|^2\},$$
	and applying Lebesgue's dominated convergence theorem, we obtain that $L_k \eta \to L \eta$ uniformly over compact subsets of $\mathbb{R}^n$, where $L$ is the operator associated with the limit $\mu$ of the measures $\mu_k$. Since $\eta$ is compactly supported in $\Omega$, we also have the bound
		$$|L_k \eta(x)| \le \frac{C}{1 + |x|^{n + 2s}}.$$
	Combining this with the growth estimate for $u_k$, it follows that
		$$|u_k L_k \eta| \le \frac{C}{1 + |x|^{n + \varepsilon}},$$
	for some $\varepsilon > 0$. Hence, by dominated convergence again,
		$$\int_{\mathbb{R}^n} u_k \, L_k \eta \, dx \to \int_{\mathbb{R}^n} u \, L \eta \, dx, \quad \text{for all } \eta \in C_0^\infty(\Omega).$$
	
	Next, since $\{a_k\}$ is uniformly bounded in $C^{0,\alpha}_{\text{loc}}(\mathbb{R}^n)$, the Arzelà-Ascoli theorem ensures the existence of a function $a \in C^{0,\alpha_1}_{\text{loc}}(\mathbb{R}^n)$, for some $\alpha_1 \in (0,\alpha)$, such that
		$$a_k \to a \quad \text{uniformly on compact subsets of } \mathbb{R}^n.$$
	Moreover, as $\{u_k\}$ is bounded in $H^1(\Omega)$, we have
		$$\nabla u_k \rightharpoonup \nabla u \quad \text{weakly in } L^2(\Omega).$$
	Consequently,
		$$\int_{\Omega} a_k(x) \nabla u_k \cdot \nabla \eta \, dx \to \int_{\Omega} a(x) \nabla u \cdot \nabla \eta \, dx, \quad \text{for all } \eta \in C_0^\infty(\Omega).$$
	Finally, since
		$$\int_{\Omega} f_k \eta \, dx \to \int_{\Omega} f \eta \, dx,$$
	we conclude that $u$ is a weak solution to the equation
		$$\mathcal{E} u \coloneqq L u - \operatorname{div}(a(x) \nabla u) = f \quad \text{in } \Omega.$$
\end{proof}

Before stating and proving the first interior regularity result, we present the key step required for its proof.

\begin{proposition}\label{prestep}
	Let $\mathcal{E}$ be an operator as defined in $(P)$, where $a(\cdot) \in C_{\mathrm{loc}}^{0,\alpha}(\mathbb{R}^n)$, for some $\alpha \in (0,1)$, satisfying \eqref{conda}, and $L$ satisfies \eqref{defL}–\eqref{condL} with $s \in (0,1)$. Suppose that $w \in C_0^\infty(\mathbb{R}^n)$ satisfies
		$$\E w = f \ \mbox{in} \ B_1,$$
	with $f \in C^\gamma(B_1)$, for some $\gamma \in (0,1)$. Assume that  $2s+\gamma+\alpha \notin \mathbb{N}$, and one of the conditions below is satisfied:
		\begin{enumerate}[$(a)$]
			\item If $\frac{1 - \alpha}{2} < s < 1$, assume that $\gamma \in (0,1)$ is such that $\lfloor 2s+\alpha+\gamma\rfloor \le 2$;
			
			\item If $0 < s \le \frac{1 - \alpha}{2}$, suppose in addition that $\gamma \in (0,1)$ is such that $2s + \gamma + \alpha \ge 1$.
		\end{enumerate}	
	Let $\gamma' \in (0, \gamma)$ satisfying $\lfloor 2s+\alpha+\gamma\rfloor < 2s+\alpha+\gamma'$, and $\gamma <2s+\gamma'$. Then the following estimate holds:
		$$[w]_{C^{2s+\alpha+\gamma}(B_{\frac{1}{2}})} \le C \Big([f]_{C^{\gamma}(B_1)} + \|w\|_{C^{2s+\alpha+\gamma'}(\mathbb{R}^n)}\Big).$$
\end{proposition}
\begin{proof}
	$(a)$ We proceed via contradiction. Assume the result fails. Then, for each $k \in \mathbb{N}$, there exist:
		\begin{itemize}
			\item a sequence of modular $\{a_k\} \subset C_{loc}^{0,\alpha}(\mathbb{R}^n)$;
			
			\item a sequence of functions $\{w_k\} \subset C_0^\infty(\mathbb{R}^n)$;
			
			\item right-hand sides $\{f_k\} \subset C^\gamma(B_1)$;
			
			\item and a family of operators $\{L_k\}$ of the form \eqref{defL} satisfying \eqref{condL}, with fixed ellipticity constants $\lambda_1, \Lambda_1 > 0$;
		\end{itemize}
	such that the following conditions are met:
		\begin{itemize}
			\item $\E_k w_k \coloneqq L_k w_k - \mbox{div}(a_k(x)\nabla w_k) = f_k$ in $B_1$;
			
			\item $[f_k]_{C^\gamma(B_1)} + \|w_k\|_{C^{2s+\alpha+\gamma'}(\mathbb{R}^n)} \le 1$;
			
			\item $\|w_k\|_{C^{2s+\alpha+\gamma}(B_{\frac{1}{2}})} \ge k$.
		\end{itemize}
Thus, the idea is use the Lemma \ref{lemcontr} to get a contradiction.

Our first goal is to show that the sequence $\{a_k\}$ is uniformly bounded. Indeed, suppose by contradiction there exists some compact set $K\subset B_1$, a subsequence, still denoted by $\{a_k\}$, and points $x_k \in K$ such that	
	$$|a_k(x_k)|\to \infty , \quad \mbox{as} \ k\to \infty.$$
Now, define $\delta_k = \|a_k\|_{L^\infty(B_{\frac{1}{2}})}^{-1/2}$ and $M_k = \|w_k\|_{C^{2s+\alpha+\gamma}(B_{\frac{1}{2}})}$. So, $\delta_k \to 0$ and $M_k \to \infty$ as $k \to \infty$. Consider also the rescaled function
	$$\widetilde{w}_k(x) \coloneqq \dfrac{w_k(x_k+\delta_k x)}{M_k}.$$
By construction,
	$$\|\widetilde{w}_k\|_{C^{2s+\alpha+\gamma'}(\mathbb{R}^n)} \le \dfrac{1}{M_k} \to 0.$$
Furthermore, the original equation $\E_k w_k = f_k$ becomes
	$$\dfrac{1}{\delta_k^{2s}} L_k \widetilde{w}_k - \dfrac{1}{\delta_k}\mbox{div}\Big(a_k(x_k+\delta_k x)\nabla \widetilde{w}_k\Big) = \dfrac{1}{M_k}f_k(x_k+\delta_k x).$$
Notice that the righ-hand side vanishes as $k\to\infty$. However, the divergence term dominates the left-hand side and behaves as $\|a_k\|_{L^\infty}^{3/2}\Delta \widetilde{w}_k$ at infinity. Then,
	$$\|a_k\|_{L^\infty}^{3/2}\Delta \widetilde{w}_k \to 0.$$
But this is possible only if $\widetilde{w}_k$ vanishes at infinity, what is impossible since $\|\widetilde{w}_k\|_{C^{2s+\gamma+\alpha}}\ge1$. That implies $\{a_k\}$ must be a uniformly bounded sequence, as claimed.
	
Throughout the proof, we will denote $\tau = \lfloor2s+\gamma +\alpha\rfloor$. Since $\tau < 2s+\gamma'+\alpha \le 2s+\gamma+\alpha$ we have
	\begin{equation}
		\sup_k \ \sup_{z \in B_{\frac{1}{2}}} \ \sup_{\rho>0} (\rho)^{\gamma'-\gamma}[w_k]_{C^{2s+\alpha+\gamma'}(B_{\rho}(z))} = \infty.
	\end{equation}
Next, we define
	$$\theta(\rho) \coloneqq \sup_k \ \sup_{z \in B_{\frac{1}{2}}} \ \sup_{\rho'>\rho} (\rho')^{\gamma'-\gamma}[w_k]_{C^{2s+\alpha+\gamma'}(B_{\rho'}(z))}.$$
The function $\theta$ satisfies:
	\begin{itemize}
		\item is monotone nonincreasing;
		
		 \item $\theta(\rho)<\infty$, for any $\rho>0$;
		
		 \item $\theta(\rho) \to \infty$, as $\rho\to0^+$;
		
		 \item $\theta(\rho)\rho^\alpha \to \infty$, as $\rho \to 0^+$.
	\end{itemize}

Now, for any positive integer $m$, by definition of $\theta(1/m)$, there exist $\rho_m' \ge 1/m$, $k_m$, and $z_m \in B_{\frac{1}{2}}$, for which
	$$(\rho_m')^{\gamma'-\gamma}[w_{k_m}]_{C^{2s+\alpha+\gamma'}(B_{\rho_m'}(z_m))}\ge \dfrac{1}{2}\theta\left(\frac{1}{m}\right)\ge \dfrac{1}{2}\theta(\rho_m').$$

To proceed, we recall the definition of the $\texttt{arg min}$ operator. Let $F : X \to \mathbb{R}$ be a real-valued function defined on a set $X$. Then,
	$$\underset{x \in X}{\texttt{arg min}}\ F(x)\coloneqq \Big\{y \in X \; \; : \;\; F(y) \le F(x), \; \mbox{for any} \ x \in X\Big\}.$$
In this context, for each $k\in\mathbb{N}$, $z \in \mathbb{R}^n$, and $\rho > 0$, we define
	$$P_{k,z,\rho}\coloneqq \underset{P \in \mathbb{P}_\tau}{\texttt{arg min}}\ \int_{B_\rho(z)} \Big(w_k(x)-P(x-z)\Big)^2dx,$$
where $\mathbb{P}_\tau$ denotes the linear space of polynomials of degree at most $\tau$ with real coefficients. In other words, $P_{k,z,\rho}(\cdot - z)$ is the polynomial of degree at most $\tau$ (in the variable $x - z$) that best approximates $w_k$ on $B_\rho(z)$ in the least-squares sense.

From now on, for simplicity we denote $P_m = P_{k_m,z_m,\rho_m'}$. We also consider the blow-up sequence
	\begin{equation}
		v_m(x) \coloneqq \dfrac{w_{k_m}(z_m+\rho_m x) - P_m(\rho_m x)}{(\rho_m)^{2s+\alpha+\gamma}\theta(\rho_m)}.
	\end{equation}
By definition of $P_m$ and changing variables we know that
	\begin{equation}\label{ort}
		\int_{B_1} v_m(x) Q(x) \ dx =0 \quad \mbox{for any} \ Q \in \mathbb{P}_\tau.
	\end{equation}
Thus, the main goal now is to prove that $\big\{v_m(\cdot+h)-v_m(\cdot)\big\}$ satisfy, for any $h\in\mathbb{R}^n$ the assumptions of Lemma \ref{lemcontr}. First of all, by same reasoning as in \cite[Proposition 3.2]{RS} we know that $v_m$ satisfies
	\begin{enumerate}[$(i)$]
		\item $[v_m]_{C^{2s+\alpha+\gamma'}(B_1)} \ge \dfrac{1}{2}$;
		
		\item $[v_m]_{C^{2s+\alpha+\gamma'}(B_R)} \le R^{\gamma-\gamma'}$, for any $R\ge1$;
		
		\item $\|v_m\|_{L^\infty(B_1)}\le C$;
		
		\item $[v_m]_{C^{\sigma}(B_R)} \le CR^{2s+\alpha+\gamma-\sigma}$, for any $\sigma \in [0,2s+\alpha+\gamma']$;
		
		\item $\displaystyle \sup_{|l|=\tau} \ \underset{B_1}{osc} \ D^l v_m \ge \frac{1}{4}.$
	\end{enumerate}
Note that $(iv)$ yields
	\begin{equation}\label{vi}
		|v_m(x+h)-v_m(x)| \le C R^{2s-\varepsilon} \quad  \mbox{for any}\ R\ge1 \ \mbox{and some}\ \varepsilon\in (0,2s+\gamma'-\gamma).
	\end{equation}
Now, let us prove a sequence of claims about $\{v_m\}$ that will help us in our objective.

\vspace{.5cm}
\noindent {\bf Claim 1:} The sequence $\{v_m\}$ converges in $C^{\frac{2s+\alpha+\gamma'+\tau}{2}}_{loc}(\mathbb{R}^n)$ to a function $v \in C^{2s+\alpha+\gamma'}_{loc}(\mathbb{R}^n)$, which satisfies the hypotheses of Corollary \ref{coroLiouville}.

This convergence is a consequence of the uniform $C^{\frac{2s+\alpha+\gamma'+\tau}{2}}$ bounds on compact subsets, given by estimate $(iv)$, along with the Arzelà–Ascoli theorem and a standard diagonalization procedure. The choice of exponent $\frac{2s+\alpha+\gamma'+\tau}{2}$ guarantees that it is less than $2s+\alpha+\gamma'$ and greater than both $\tau$ and $2s$.

In addition, by taking the limit as $m \to \infty$ in $(iv)$ for some $\sigma \in (\alpha+\gamma,2s+\alpha+\gamma']$ (this choice is possible because $\gamma < 2s+\gamma'$), we deduce
	$$[v]_{C^\sigma(B_R)}\le C R^{\gamma_0}, \quad \mbox{for any} \ R\ge1,$$
where $\gamma_0 = 2s+\alpha + \gamma - \sigma < 2s$. This shows that $v$ satisfies the required growth condition in Corollary \ref{coroLiouville}, completing the proof of the claim.

\vspace{.5cm}
\noindent {\bf Claim 2:} $\{v_m\}$ is a bounded sequence in $H^1(B_1)$.

First, by $(iii)$ we know that $\|v_m\|_{L^2(B_1)} \le C$. For the gradient, since $2s+\alpha+\gamma'> \lfloor 2s+\alpha+\gamma\rfloor\ge1$ (since $2s+\alpha>1$ by hypothesis), by $(iv)$, choosing $\sigma=1$ and $R=1$, we get
	$$\sup_{\parbox{1.5cm}{\tiny $x,y \in B_1 \\ x \neq y$}} \dfrac{|v_m(x)-v_m(y)|}{|x-y|} = [v_m]_{C^1(B_1)} \le C,$$
with $C>0$ independent of $m$. Thus,
	$$|\nabla v_m (x)| = \lim_{h\to0} \dfrac{|v_m(x+h)-v_m(x)|}{|h|} \le C,$$
which implies that $\|\nabla v_m\|_{L^\infty(B_1)}\le C$, and then  $\|\nabla v_m\|_{L^2(B_1)} \le C$, and the boundedness in $H^1(B_1)$ follows.

\vspace{.5cm}
\noindent {\bf Claim 3:} $\displaystyle \Big|\E_{k_m} \Big(\big[v_m(\cdot+h)-v_m(\cdot)\big](x)\Big)\Big| \le \dfrac{1}{\theta(\rho_m)\rho_m^\alpha},$ whenever $|x| \le \dfrac{1}{2\rho_m}.$

Indeed, first note that, as usual, we define increment
	$$\delta(\varphi,x,y) \coloneqq \varphi(x+y) + \varphi(x-y) - 2\varphi(x),$$
for any map $\varphi$. Since $\tau \le 2$ (by hypothesis), it follows that
	\begin{equation}\label{delta33}
		\delta(P,x+h,y) - \delta(P,x,y) = 0,
	\end{equation}
for any $P \in \mathbb{P}_\tau$ and any $x,y,z \in \mathbb{R}^n$. Now, recalling that $w_k$ satisfies $\E_k w_k = f_k$ in $B_1$ and $[f_k]_{C^\gamma(B_1)}\le1$, we find
	$$\Big|\E_k w_k(\overline{x}+\overline{h}) - \E_k w_k(\overline{x})\Big|\le |\overline{h}|^\gamma,$$
for any $\overline{x}\in B_{\frac{1}{2}}(z)$ and $\overline{h} \in B_{\frac{1}{2}}$. By setting $\overline{h} = \rho_m h$ and $\overline{x} = z_m + \rho_m x$, we obtain
	\begin{equation}\label{*1}
		\Big|\E_{k_m} w_{k_m}\big(z_m+\rho_m(x+h)\big) - \E_{k_m} w_{k_m}\big(z_m+\rho_m x\big)\Big| \le (\rho_m)^\gamma |h|^\gamma.
	\end{equation}
On the other hand, using the definition of $v_m$, we can write
	$$w_{k_m} \Big(z_m + \rho_m(x+h)\Big) = (\rho_m)^{2s+\alpha+\gamma} \theta(\rho_m)v_m(x+h) + P_m(\rho_m(x+h)),$$
and similarly,
	$$w_{k_m}\big(z_m+\rho_m x\big) = (\rho_m)^{2s+\alpha+\gamma}\theta(\rho_m)v_m(x) + P_m(\rho_m x).$$
Hence, after a change of variables, we find
	{\footnotesize\begin{eqnarray}\label{*2}
		L_{k_m} w_{k_m}\Big(z_m+\rho_m(x+h)\Big) - L_{k_m}w_{k_m}\big(z_m+\rho_m x\big)&=& L_{k_m} \Big((\rho_m)^{2s+\alpha+\gamma} \theta(\rho_m)\big(v_m(x+h)-v_m(x)\big)\Big) \nonumber \\
		&& + L_{k_m}\Big(P_m\big(\rho_m(x+h)\big) - P_m\big(\rho_m x\big)\Big) \nonumber \\
		&\stackrel{\eqref{delta33}}{=}& \dfrac{1}{(\rho_m)^{2s}} L_{k_m} \Big((\rho_m)^{2s+\alpha+\gamma} \theta(\rho_m)\big(v_m(\cdot+h)-v_m\big)(x)\Big). \nonumber \\
	\end{eqnarray}}
Similarly, we have
	\begin{eqnarray}\label{*3}
		\mbox{div}\Big(a_{k_m}(x) \nabla w_{k_m}(z_m\rho_m(x+h))\Big) &-& \mbox{div} \Big(a_{k_m}(x)\nabla w_{k_m}(z_m+\rho_m x)\Big) \nonumber \\
		&=& \mbox{div} \Big((\rho_m)^{2s+\alpha+\gamma}\theta(\rho_m)a_{k_m}(x)\big(\nabla v_m(x+h) - \nabla v_m(x)\big)\Big) \nonumber \\
		&=& \dfrac{1}{\rho_m^2} \mbox{div} \Big((\rho_m)^{2s+\alpha+\gamma}\theta(\rho_m)a_{k_m}(x)\big(\nabla v_m (\cdot+h)-\nabla v_m\big)(x)\Big). \nonumber \\
	\end{eqnarray}
By combining equations \eqref{*1}, \eqref{*2}, and \eqref{*3}, we conclude
	$$\dfrac{1}{(\rho_m)^{2s}} \Big|\E_{k_m} \Big((\rho_m)^{2s+\alpha+\gamma}\theta(\rho_m)\big(v_m(\cdot+h) - v_m\big)(x)\Big)\Big|\le (\rho_m)^\gamma |h|^\gamma,$$
whenever $|x|\le \frac{1}{2\rho_m}$. Here, we use that $\rho_m<1$ for $m$ large enough. And this complete the proof of Claim 3.

In summary, we began by noting that the sequence $\{a_k\}$ is uniformly bounded. From Claim 2, the sequence $\{v_m(\cdot+h) - v_m\}$ is bounded in $H^1(B_1)$, while Claim 1 ensures that it converges uniformly on compact sets to $v(\cdot+h) - v$. Moreover, the uniform $C^\gamma$ bound on $\{f_{k_m}\}$ ensures that these functions converge uniformly in $B_1$ to some limit $f$. Finally, using the compactness of probability measures on the sphere, we extract a subsequence of spectral measures $\{\mu_{k_m}\}$ converging to a spectral measure $\mu$ associated with an operator $L$ of the form \eqref{defL}–\eqref{condL}. These properties, along with \eqref{vi}, enable the application of Lemma \ref{lemcontr} to the sequence $\{v_m(\cdot + h) - v_m\}$.

Consequently, taking the limit as $m \to \infty$ in Claim 3, by applying Lemma \ref{lemcontr} and since $\theta(\rho)\rho^\alpha \to\infty$, as $\rho \to 0^+$, we deduce
	$$\E \Big(\big(v(\cdot+h)-v\big)(x)\Big) = 0, \quad \mbox{in whole} \ \mathbb{R}^n,$$
where $\E(\cdot) \coloneqq L(\cdot) - \mbox{div}(a(x)\nabla (\cdot))$. Since Claim 1 guarantees that $v$ satisfies the assumptions of Corollary \ref{coroLiouville}, it follows that $v$ must be a polynomial of degree at most $\tau$. However, passing to the limit in \eqref{ort}, we find that $v$ is orthogonal to every such polynomial in $B_1$, and hence $v \equiv 0$. This is incompatible with the normalization condition in item $(i)$ in the limit, yielding a contradiction.

\vspace{.5cm}

\noindent $(b)$ Observe that in the preceding proof, for our case involving a double-phase structure, the primary effect of the interaction between $s$ and the H\"{o}lder exponents appears in the boundedness of the sequence ${v_m}$ in the space $H^1(B_1)$, as shown in Claim 2, and in the decay estimates for the operators $\E_{k_m}$, as discussed in Claim 3. Specifically, in item $(a)$, since $2s + \alpha > 1$, we may proceed as in Claim 2 without further assumptions. However, to apply Claim 3, it is necessary to assume that $\lfloor 2s + \alpha + \gamma \rfloor \le 2$. On the other hand, in item $(b)$, since $2s + \alpha \le 1$, we automatically have $\lfloor 2s + \alpha + \gamma \rfloor \le 2$, and thus Claim 3 holds {\it ipsis litteris}. In this case, however, we must impose $2s + \gamma + \alpha > 1$ in order to ensure Claim 2 applies. Therefore, taking these considerations into account, the previous arguments can be repeated, and item $(b)$ follows.
\end{proof}


Now we are able to prove the first regularity result as follows.

\begin{proof}[Proof of Theorem \ref{reg1}]
	We will prove only under condition $(a)$, since the item $(b)$ follows by same reasoning. As before, consider $\tau = \lfloor 2s+\alpha+\gamma\rfloor$, and $\gamma'\in(0,1)$ be such that $\tau <2s+\alpha+\gamma'$, and $\gamma <2s+\gamma'$. Now, let $w \in C_0^\infty(\mathbb{R}^n)$ be such that $\E w = f$ in $B_1$. Take a cutoff function $\eta \in C_0^\infty(B_2)$ such that $\eta \equiv 1$ in $B_{\frac{3}{2}}$. Note that
	$$\E (w\eta) = \E(w\eta -w) + f \quad \mbox{in} \ B_1.$$
Then, since $w\eta-w$ vanishes in $B_{\frac{3}{2}}$, by Proposition \ref{prestep}
	\begin{eqnarray}\label{38}
		[w\eta]_{C^{2s+\alpha+\gamma}(B_{\frac{1}{2}})} &\le& C \Big([f+\E(w\eta-w)]_{C^\gamma(B_1)} + \|w\eta\|_{C^{2s+\alpha+\gamma'}(B_2)}\Big) \nonumber \\
		&\le& C \Big([f]_{C^\gamma(B_1)}+[L(w\eta-w)]_{C^\gamma(B_1)} + \|w\eta\|_{C^{2s+\alpha+\gamma'}(B_2)}\Big).
	\end{eqnarray}
Set $\varphi \coloneqq w\eta -w$. The product rule for H\"{o}lder seminorms yields $[\varphi]_{C^\gamma(\mathbb{R}^n)} \le C \|w\|_{C^\gamma(\mathbb{R}^n)}$. Thus,
	\begin{eqnarray*}
		|L\varphi(x)-L\varphi(y)| &\le& \int_{S^{n-1}} \int_{\frac{1}{2}}^{+\infty} \dfrac{|\varphi(x+\theta r) - \varphi(y+\theta r)|+|\varphi(y-\theta r)- \varphi(x-\theta r)|}{|r|^{1+2s}} \ dr \ d\mu(\theta) \\
		&\le& 2[\varphi]_{C^\gamma(\mathbb{R}^n)} \ |x-y|^\gamma \int_{S^{n-1}} d\mu \int_{\frac{1}{2}}^{+\infty} r ^{-1-2s}dr \\
		&\le& \dfrac{2^{2s+1}}{s} \Lambda [\varphi]_{C^\gamma(\mathbb{R}^n)} \ |x-y|^\gamma,
	\end{eqnarray*}
for any $x,y \in B_1$. Then,
	\begin{equation}\label{39}
		[L\varphi]_{C^\gamma(B_1)} = \sup_{\substack{x,y \in B_1 \\ x \neq y}} \dfrac{|L\varphi(x)-L\varphi(y)|}{|x-y|^\gamma} \le C[\varphi]_{C_\gamma(\mathbb{R}^n)} \le C\|w\|_{C^\gamma(\mathbb{R}^n)}.
	\end{equation}
By combining \eqref{38} and \eqref{39}, we get
	$$[w]_{C^{2s+\alpha+\gamma}(B_\frac{1}{2})} \le C\Big([f]_{C^\gamma(B_1)} + \|w\|_{C^\gamma(\mathbb{R}^n)} + \|w\|_{C^{2s+\alpha+\gamma'}(B_2)}\Big).$$
Therefore, by standard arguments (see for instance \cite{GT} or \cite{S}), we have
	$$\|w\|_{C^{2s+\alpha+\gamma}(B_\frac{1}{2})} \le C \Big(\|f\|_{C^\gamma(B_1)}+\|w\|_{C^\gamma(\mathbb{R}^n)}\Big).$$
Finally, to get the previous result for any $u \in C^\gamma$, consider $\eta_\varepsilon \in C_0^\infty(B_\varepsilon)$ a mollifier, and $u_\varepsilon \coloneqq u * \eta_\varepsilon$. Then,
	$$\E u_\varepsilon = \E u * \eta_\varepsilon = f * \eta_\varepsilon,$$
in $B_1^\varepsilon = \{x \in B_1 \; \; : \;\; \mbox{dist}(x,S^{n-1})>\varepsilon\}$. Therefore, by approximation arguments
	 $$\|u\|_{C^{2s+\alpha+\gamma}(B_\frac{1}{2})} \le C \Big(\|f\|_{C^\gamma(B_1)}+\|u\|_{C^\gamma(\mathbb{R}^n)}\Big),$$
as desired.
\end{proof}

\section{Interior regularity (part two)}\label{Sec4}

In this section we will prove the regularity estimate with a $L^\infty$ right-hand side. Similar as before, we prove first a preliminary result.

\begin{proposition}\label{presteps>}
	Let $\E$ be an operator as defined in $(P)$, where $a(\cdot) \in C_{loc}^{0,\alpha}(\mathbb{R}^n)$, for some $\alpha\in(0,1)$, satisfying \eqref{conda}, and $L$ satisfies \eqref{defL} and \eqref{condL}, with $\frac{1-\alpha}{2} < s <1$ and $2s+\alpha \notin \mathbb{N}$. Let $\gamma \in (0,2s+\alpha)$ be such that
		$$\lfloor 2s+\alpha \rfloor < \gamma < 2s+\alpha.$$
	Assume $w \in C_0^{\infty}(\mathbb{R}^n)$ satisfying
		$$\E w = f, \quad \mbox{in} \ B_1,$$
	with $f \in L^\infty(B_1)$. Then,
		$$[w]_{C^{2s+\alpha}(B_{\frac{1}{2}})} \le C \Big(\|f\|_{L^\infty(B_1)} + \|w\|_{C^\gamma(\mathbb{R}^n)}\Big).$$
\end{proposition}
\begin{proof}
	We follow the same reasoning as in Proposition \ref{prestep}, highlighting only the key differences. Suppose the claimed estimate fails. Then there exist sequences $a_k\in C^{0,\alpha}$, $w_k\in C^\infty_0(\mathbb{R}^n)$, $f_k\in L^\infty(B_1)$, and $L_k$ of type \eqref{defL} satisfying \eqref{condL} (sharing the same ellipticity constants $\lambda_1,\Lambda_1$), such that
		$$\E_k w_k=f_k\quad\text{in }B_1,\quad
		\|f_k\|_{L^\infty(B_1)}+\|w_k\|_{C^\gamma(\mathbb{R}^n)}\le1,
		\quad
		[w_k]_{C^{2s+\alpha}(B_{1/2})}\to\infty.$$
	Similarly as in Proposition \ref{prestep} we know that $\{a_k\}$ is uniformly bounded. Define
		$$\tau\coloneqq\lfloor2s+\alpha\rfloor,\quad
		\beta\coloneqq2s+\alpha-\gamma>0,$$
	and for $\rho>0$ let the quantity
		$$\theta(\rho)  =\sup_k\sup_{z\in B_{1/2}}\sup_{\rho'>\rho}(\rho')^{-\beta}\,[w_k]_{C^{2s+\alpha}(B_{\rho'}(z))}.$$
	Then $\theta$ is nonincreasing, finite for each $\rho>0$, and
	$\theta(\rho)\,\rho^\alpha\to\infty$ as $\rho\to0^+$. For each $m\in\mathbb{N}$, choose
		$$0<\rho'_m\le\tfrac1m,\quad z_m\in B_{1/2},\quad k_m\in\mathbb{N},$$
	such that
		\begin{equation}\label{rmt}
			(\rho'_m)^{-\beta}\,[w_{k_m}]_{C^{2s+\alpha}(B_{\rho'_m}(z_m))}
			\ge \dfrac{1}{2}\theta(\rho'_m).
		\end{equation}
	Define $P_m = P_{k_m,z_m,\rho_m'}$ as in the Proposition \ref{prestep}, and consider the blow-up sequence
		$$v_m(x) \coloneqq \dfrac{w_{k_m}(z_m+\rho_m' x) - P_m(\rho_m' x)}{(\rho_m')^{2s+\alpha}\theta(\rho_m')}.$$
	Note that, for any $m\ge1$
		\begin{equation}\label{intvQ}
			\int_{B_1} v_m(x)Q(x)\ dx=0, \quad \mbox{for any} \ Q \in \mathbb{P}_\tau.
		\end{equation}
	By \eqref{rmt}, the same reasoning as in Proposition \ref{prestep} yields
		\begin{equation}\label{belowinf}
			[v_m]_{C^\gamma(B_1)} \ge \dfrac{1}{2}, \quad \mbox{for any} \ m\ge1,
		\end{equation}
		\begin{equation}\label{uniformB}
			[v_m]_{C^\sigma(B_R)} \le C R^{2s+\alpha-\sigma}, \quad \mbox{for any} \ R \ge1 \ \mbox{and} \ \sigma\in[0,\gamma],
		\end{equation}
	as well as,	
		\begin{equation}
			|v_m(x)| \le C R^{2s-\varepsilon} \quad  \mbox{for any}\ R\ge1 \ \mbox{and some}\ \varepsilon>0.
		\end{equation}
	
	By Arzelà–Ascoli and diagonalization, a subsequence converges $v_m\to v$ in $C^{\gamma'}_{\mathrm{loc}}(\mathbb{R}^n)$ for all $\gamma'<\gamma$. Moreover, passing to the limit of $m\to\infty$ in \eqref{uniformB} with $\sigma=\gamma$, we find
		$$[v]_{C^\gamma(B_R)}\le CR^{2s+\alpha-\gamma}, \quad R\ge1.$$
	Since $1<2s+\alpha$ and $\lfloor2s+\alpha \rfloor<\gamma$, then $2s+\alpha-\gamma <2s$, and this implies that $v$ satisfies the assumptions of Corollary \ref{coroLiouville}.
	
	Further, in a same way as in Claim 2 of Proposition \ref{prestep}, but now using \eqref{belowinf}, \eqref{uniformB}, and the fact that $\gamma >1$, we can see that $\{v_m\}$ is a bounded sequence in $H^1(B_1)$.
	
	Finally, following the same reasoning of Claim 3 in Proposition \ref{prestep}, but now using that $\|f_k\|_{L^\infty(B_1)} \le 1$, we can prove that
		\begin{equation}\label{Emlimit}
			\Big|\E_{k_m} \Big(\big[v_m(\cdot+h)-v_m\big](x)\Big)\Big| \le \dfrac{2}{\theta(\rho_m'){(\rho_m')^\alpha}},
		\end{equation}
	whenever $|x| \le \dfrac{1}{2\rho_m}.$
	
	In summary, we have that
		\begin{itemize}
			\item the sequence $\{a_k\}$ is uniformly bounded;
			
			\item the sequence $\{v_m(\cdot+h) - v_m\}$ is bounded in $H^1(B_1)$, and converges uniformly on compact sets to $v(\cdot+h) - v$, with $v$ satisfying the assumptions of Corollary \ref{coroLiouville};
			
			\item $|v_m(x+h)-v_m(x)| \le C R^{2s-\varepsilon}$, for any $R\ge1$ and $h \in \mathbb{R}^n$.
		\end{itemize}
	Thus, passing to the limit of $m \to \infty$ in \eqref{Emlimit}, by applying Lemma \ref{lemcontr} we get
		$$\E \Big(\big(v(\cdot+h)-v\big)(x)\Big) = 0, \quad \mbox{in whole} \ \mathbb{R}^n,$$
	where $\E$ is in the form $L - \mbox{div}(a(x)\nabla \cdot)$. Since $v$ satisfies the assumptions of Corollary \ref{coroLiouville}, it follows that $v$ must be a polynomial of degree at most $\tau$. However, passing to the limit in \eqref{intvQ}, we find that $v$ is orthogonal to every such polynomial in $B_1$, and hence $v \equiv 0$. This is incompatible with the normalization condition in \eqref{belowinf} in the limit, yielding a contradiction.
\end{proof}

Note that in the previous proposition, the choice of $\gamma$  is feasible precisely due to the strict inequality $\frac{1-\alpha}{2} < s <1$ holds. For the borderline case, which means $s = \frac{1-\alpha}{2}$, we establish the following regularity estimate.

\begin{proposition}\label{prestepborder}
	Let $\E$ be an operator as defined in $(P)$, where $a(\cdot) \in C_{loc}^{0,\alpha}(\mathbb{R}^n)$, for some $\alpha\in(0,1)$, satisfying \eqref{conda}, and $L$ satisfies \eqref{defL} and \eqref{condL}, with $s=\frac{1-\alpha}{2}$. Suppose $w \in C_0^{\infty}(\mathbb{R}^n)$ solves
		$$\E w = f, \quad \mbox{in} \ B_1,$$
	with $f \in L^\infty(B_1)$. Then, for any $\varepsilon \in (0,1)$
		$$[w]_{C^{2s+\alpha-\varepsilon}(B_{\frac{1}{2}})} \le C \Big(\|f\|_{L^\infty(B_1)} + \|w\|_{L^\infty(\mathbb{R}^n)}\Big).$$
\end{proposition}
\begin{proof}
	We adapt the argument from Proposition \ref{presteps>}, with two key adjustments, namely: Set $\tau =0$, and $\beta = 2s+\alpha-\varepsilon$. For the blow-up sequence $\{v_m\}$ we will be able to prove that
		$$[v_m]_{C^\sigma(B_R)} \le C R^{2s+\alpha-\varepsilon}, \quad \mbox{for any} \ R\ge 1 \ \mbox{and} \ \sigma \in [0,2s+\alpha].$$
	This ensures $\{v_m\}$ remains bounded in $H^1(B_1)$, analogous to the earlier reasoning. The remaining steps follow verbatim from Proposition \ref{presteps>}, as the structure of the contradiction argument and limiting process remain unchanged.
\end{proof}

Now we are able to prove the second regularity result as follows.

\begin{proof}[Proof of Theorem \ref{reg2}]
	We begin by considering the case $s \neq \frac{1 - \alpha}{2}$. Choose any \linebreak $\gamma \in \big(\lfloor 2s + \alpha \rfloor, 2s + \alpha\big)$, and let $w \in C_0^\infty(\mathbb{R}^n)$ be a smooth function satisfying
		$$\E w = f, \quad \mbox{in} \ B_1.$$
	Let $\eta \in C_0^\infty(B_2)$ be a cutoff function such that $\eta \equiv 1$ in $B_{\frac{3}{2}}$. Then, using the linearity of the operator, we obtain
		$$\E (w\eta) = \E \big(w\eta -w\big) + f, \quad \mbox{in} \ B_1.$$
	Thus, since $w\eta - w$ vanishes in $B_{\frac{3}{2}}$, Proposition \ref{presteps>} yields
		\begin{eqnarray}\label{weta}
			[w\eta]_{C^{2s+\alpha}(B_{\frac{1}{2}})} &\le& C \Big(\|\E(w\eta -w)+f\|_{L^\infty(B_1)} + \|w\eta\|_{C^\gamma(B_2)}\Big) \nonumber \\
			&\le& C \Big(\|L(w\eta -w)\|_{L^\infty(B_1)} + \|f\|_{L^\infty(B_1)} + \|w\eta\|_{C^\gamma(B_2)}\Big).
		\end{eqnarray}
	To control the nonlocal term $|L(w\eta - w)|{L^\infty(B_1)}$, define $\phi \coloneqq w\eta - w$, and notice that
		\begin{equation}\label{phi}
			\|\phi\|_{L^\infty(\mathbb{R}^n)} \le \|w\|_{L^\infty(\mathbb{R}^n)}.
		\end{equation}
	Further, for any $x \in B_1$
		\begin{eqnarray}\label{Lphi}
			|L\phi(x)| &\le& \int_{S^{n-1}} \int_{-\infty}^{+\infty} \dfrac{|\phi(x+\theta r) + \phi(x-\theta r) - 2\phi(x)|}{|r|^{1+2s}}\ dr \ d\mu(\theta) \nonumber \\
			&=& \int_{S^{n-1}} \int_{\frac{1}{2}}^{+\infty} \dfrac{|\phi(x+\theta r) + \phi(x-\theta r)|}{|r|^{1+2s}}\ dr \ d\mu(\theta) \nonumber \\
			&\le& C \int_{S^{n-1}} \int_{\frac{1}{2}}^{+\infty} \dfrac{|\phi(x+\theta r)|}{|r|^{1+2s}}\ dr \ d\mu(\theta) \nonumber \\
			&\le& C \|\phi\|_{L^\infty(\mathbb{R}^n)} \int_{S^{n-1}} \int_{\frac{1}{2}}^{+\infty} r^{-(1+2s)} \ dr \ d\mu(\theta) \nonumber \\
			&\le& C \Lambda_1 \|\phi\|_{L^\infty(\mathbb{R}^n)}.
		\end{eqnarray}
	Plugging \eqref{Lphi} combined with \eqref{phi} into \eqref{weta}, we obtain
		\begin{eqnarray}
			[w]_{C^{2s+\alpha}(B_{\frac{1}{2}})} &=& [w\eta]_{C^{2s+\alpha}(B_{\frac{1}{2}})} \nonumber \\
			&\le& C \Big(\|f\|_{L^\infty(B_1)} + \|w\|_{L^\infty(B_1)} + \|w\|_{C^\gamma(B_2)}\Big), \nonumber
		\end{eqnarray}
	where we used $\|w\eta\|_{C^\gamma(B_2)} \le C\|w\|_{C^\gamma(B_2)}$, for $C>0$ depending on $\eta$ fixed. By standard interpolation and the smoothness of the cutoff, it follows that	
		$$\|w\|_{C^{2s+\alpha}(B_{\frac{1}{2}})} \le C \Big(\|w\|_{L^\infty(\mathbb{R}^n)} + \|f\|_{L^\infty(B_1)}\Big).$$
	
	To obtain the result for a general bounded weak solution $u \in L^\infty(\mathbb{R}^n)$, we approximate $u$ by smooth functions as in the proof of Theorem~\ref{reg1}, and pass to the limit using standard stability arguments. This yields
		$$\|u\|_{C^{2s+\alpha}(B_{\frac{1}{2}})} \le C \Big(\|u\|_{L^\infty(\mathbb{R}^n)} + \|f\|_{L^\infty(B_1)}\Big),$$
	as desired.
	
	Finally, in the critical case $s = \frac{1 - \alpha}{2}$, the same line of reasoning applies, except that we invoke Proposition \ref{prestepborder} in place of Proposition \ref{presteps>}. The argument then leads to a slightly weaker estimate due to the borderline nature of the regularity, yielding the desired conclusion with an $\varepsilon$-loss in the exponent. This completes the proof.
\end{proof}

\section{Regularity up to the boundary}\label{Sec5}

In this section we will consider $a(x) \equiv 1$ and prove the regularity up to the boundary of weak solutions to
	$$\left\{
	\begin{array}{rcccl}
		\mathcal{E} u & =& f & \text{in} & \Omega; \\
		u & = & 0 & \text{in} & \mathbb{R}^n \setminus \Omega.
	\end{array}
	\right.$$
To get the H\"{o}lder regularity, the main idea is to prove the following $W^{2,p}$-regularity theorem:

\begin{theorem}\label{W2p-reg}
	Let $\Omega$ be a $C^{1,1}$ domain in $\mathbb{R}^n$ and $s \in (0,1)$. Suppose that $L$ is any operator of the form \eqref{defL} and \eqref{condL}, and $a\equiv 1$. If $f \in L^p(\Omega)$, where $p$ satisfies
		\begin{equation} \label{rangep}
		\left\{
		\begin{array}{rc}
			1<p<\infty, & \mbox{if} \ s\in (0,\frac{1}{2}]; \\
			n<p<\dfrac{n}{2s-1}, & \mbox{if}\ s\in (\frac{1}{2},1),
		\end{array}
		\right.
		\end{equation}
	then, the problem $\mathcal{E} u = f$ in $\Omega$ has a unique solution $u \in W^{2,p}(\Omega)\cap W^{1,p}_0(\Omega)$. Furthermore,
		$$\|u\|_{2,p} \le C(p,s,\Lambda_1,n,\Omega)\Big(\|u\|_p + \|f\|_p\Big).$$
\end{theorem}

The proof of the previous theorem is inspired by \cite{SVWZ}, and follows the next steps:
	\begin{itemize}
		\item Step one: Obtain $L^p$-estimates on the operator $L$;
		
		\item Step two: Apply the fixed point theorem to deduce that the problem
			$$-\Delta u + \lambda u = f - Lu,$$
			has a unique solution $u \in W^{2,p}(\Omega) \cap W_0^{1,p}(\Omega)$ for $\lambda >0$ large;
			
		\item Step three: Employ the maximum principle (see Theorem \ref{maxprin}) and the bootstrap method to conclude.
	\end{itemize}
\medskip	
\begin{tabular}{|l|}
	\hline
	Step one\\
	\hline
\end{tabular}

From this point onward, we assume that $\Omega$ is a $C^{1,1}$ domain in $\mathbb{R}^n$. Let  $E: W^{2,p}(\Omega) \to W^{2,p}(\mathbb{R}^n)$ denote the extension operator provided by \cite[Theorem 7.25]{GT}, which satisfies $Eu=u$ in $\Omega$ and $\|Eu\|_{2,p}\le C(\Omega)\|u\|_{2,p}$.

For clarity, we divide the argument into two lemmas according to the range of $s$, namely $s \in (0,\frac{1}{2}]$ and $s \in (\frac{1}{2},1)$.

\begin{lemma}\label{L<12Lp}
	Let $\Omega$ be a $C^{1,1}$ domain in $\mathbb{R}^n$ and $s \in (0,\frac{1}{2}]$. If $L$ is any operator of the form \eqref{defL} and \eqref{condL}, then for any $p>1$ and $u \in W^{2,p}(\Omega) \cap W_0^{1,p}(\Omega)$, the operator $L$ satisfies
		$$\|Lu\|_p \le C \Big(o(\varepsilon) \|u\|_{2,p} + \tau(\varepsilon)\|u\|_p\Big), \quad \mbox{for any} \ \varepsilon>0,$$
	where the constant $C>0$ depends only on $p,s,\Lambda,n$, and $\Omega$, and $o(\varepsilon) \to 0$ and $\tau(\varepsilon)$ is unbounded, both as $\varepsilon\to 0$.
\end{lemma}
\begin{proof}
	Let us start with the case $s \in (0,\frac{1}{2})$. Consider $u \in W^{2,p}(\Omega)\cap W_0^{1,p}(\Omega)$, and note that
		\begin{eqnarray}\label{AB}
			Lu(x) &=& \int_{S^{n-1}} \int_{\{|r|> 1\}} \dfrac{u(x+\theta r)+u(x-\theta r)-2u(x)}{|r|^{1+2s}}\ dr \ d\mu \nonumber \\
			&& + \int_{S^{n-1}} \int_{\{|r|\le 1\}} \dfrac{u(x+\theta r)+u(x-\theta r)-2u(x)}{|r|^{1+2s}}\ dr \ d\mu \nonumber \\
			&\eqcolon& I +II.
		\end{eqnarray}
	We may estimate, by using H\"{o}lder inequality with $1/p+1/q=1$
		\begin{eqnarray*}
			\|I\|_p &\le& \left(\int_\Omega \left[\int_{S^{n-1}} \int_{\{|r|> 1\}} \dfrac{|u(x+\theta r)| + |u(x-\theta r)| + 2|u(x)|}{|r|^{1+2s}}dr \ d\mu\right]^pdx\right)^{1/p} \\
			&=& \left(\int_{\mathbb{R}^n} \left[\int_{S^{n-1}} \int_{\{|r|> 1\}} \dfrac{|Eu(x+\theta r)| + |E(x-\theta r)| + 2|Eu(x)|}{|r|^{\frac{1+2s}{p}} \cdot|r|^{\frac{1+2s}{q}}}dr \ d\mu\right]^pdx\right)^{1/p} \\
			&\le& C(p) \left(\int_{\mathbb{R}^n}\int_{S^{n-1}} \int_{\{|r|> 1\}}\dfrac{|Eu(x+\theta r)|^p + |E(x-\theta r)|^p + 2^p|Eu(x)|^p}{|r|^{1+2s}}\ dr \ d\mu \ dx\right)^{1/p}  \\
			&& \hspace{8cm} \times \left(\int_{\{|r|> 1\}} \dfrac{1}{|r|^{1+2s}} dr\right)^{1/q}.
		\end{eqnarray*}
	Hence, by Fubini theorem and properties of the Extension operator, we get
		\begin{eqnarray}\label{estI}
			\|I\|_p &\le& C(p,\Lambda_1) \|E u\|_p \left(\int_{\{|r|> 1\}} \dfrac{1}{|r|^{1+2s}} dr\right) \nonumber\\
			&\le& C(p,\Lambda_1,s) \|E u\|_p \nonumber\\
			&\le& C(p,\Lambda_1,s,\Omega) \|u\|_{1,p}.
		\end{eqnarray}
	To estimate $\|II\|_p$, since $u = 0$ in $\mathbb{R}^n \setminus \Omega$, we may rewrite $II$ as
		\begin{eqnarray*}
			II &=& \int_{S^{n-1}} \int_{\{|r|\le 1\}} \dfrac{u(x+\theta r)+u(x-\theta r)-2u(x)}{|r|^{1+2s}}\ dr \ d\mu \\
			&=& \int_{S^+_1\cap S_1^-} \int_{\{|r|\le 1\}} \dfrac{u(x+\theta r) + u(x-\theta r)-2u(x)}{|r|^{1+2s}}\ dr \ d\mu \\
			&& + \int_{S^+_1\cap S_2^-} \int_{\{|r|\le 1\}} \dfrac{u(x+\theta r)-2u(x)}{|r|^{1+2s}} dr  d\mu + \int_{S^+_2\cap S_1^-} \int_{\{|r|\le 1\}} \dfrac{u(x-\theta r)-2u(x)}{|r|^{1+2s}}dr  d\mu \\
			&&+ 2\int_{S_2^+\cap S^-_2} \int_{\{|r|\le 1\}} \dfrac{-u(x)}{|r|^{1+2s}}\ dr \ d\mu \\
			&\eqcolon& A+B+C+D,
		\end{eqnarray*}
	where
		$$S_1^\pm \coloneqq \big\{ \theta \in S^{n-1} \:\: : \:\: x \pm \theta r \in \Omega\big\}\quad \mbox{and} \quad S_2^\pm \coloneqq \big\{ \theta \in S^{n-1} \:\: : \:\: x \pm \theta r \notin \Omega\big\}.$$
	As before, considering $1/p+1/q=1$, H\"{o}lder inequality, Fubini theorem and properties of the Extension operator yield
		\begin{eqnarray}
			&&\left(\int_\Omega \left[\int_{S^+_1\cap S^-_1} \int_{\{|r|\le 1\}} \dfrac{u(x\pm\theta r)-u(x)}{|r|^{1+2s}}\ dr \ d\mu\right]^p dx\right)^{1/p}\nonumber \\
			&&\le \left(\int_{\mathbb{R}^n} \left[\int_{S^{n-1}} \int_{\{|r|\le 1\}} \int_{0}^{1}|\nabla Eu(x\pm t\theta r)|\dfrac{1}{|r|^{2s}}\ dt \ dr \ d\mu\right]^p dx\right)^{1/p}\nonumber\\
			&& \le \left(\int_{\mathbb{R}^n} \int_{S^{n-1}} \int_{\{|r|\le 1\}} |\nabla Eu(x\pm t\theta r)|^p \dfrac{1}{|r|^{2s}} \ dr\ d\mu \ dx\right)^{1/p} \left(\int_{S^{n-1}} \int_{\{|r|\le 1\}} \dfrac{1}{|r|^{2s}} \ dr \ d\mu \right)^{1/q}\nonumber \\
			&& \le \|Eu\|_{1,p} \left(\int_{S^{n-1}} \int_{\{|r|\le 1\}} \dfrac{1}{|r|^{2s}} \ dr \ d\mu \right)\nonumber \\
			&& \le C(p,\Lambda_1,s,\Omega) \|u\|_{1,p}, \nonumber
		\end{eqnarray}
	which implies
		\begin{equation}\label{A}
			\|A\|_p \le C(p,\Lambda_1,s,\Omega) \|u\|_{1,p}.
		\end{equation}
	Similarly, we have
		$$\left(\int_\Omega \left[\int_{S_1^+ \cap S^-_2} \int_{\{|r|\le 1\}} \dfrac{u(x+\theta r)-u(x)}{|r|^{1+2s}}\ dr \ d\mu\right]^p dx\right)^{1/p} \le C(p,\Lambda_1,s,\Omega) \|u\|_{1,p},$$
	and
		$$\left(\int_\Omega \left[\int_{S_2^+ \cap S^-_1} \int_{\{|r|\le 1\}} \dfrac{u(x-\theta r)-u(x)}{|r|^{1+2s}}\ dr \ d\mu\right]^p dx\right)^{1/p} \le C(p,\Lambda_1,s,\Omega) \|u\|_{1,p},$$
	Thus, in order to complete the estimates for $B$ and $C$, and to obtain the estimate for $D$, it remains to analyze the integrals
		$$\int_{S_1^+\cap S^-_2} \int_{\{|r|\le 1\}} \dfrac{-u(x)}{|r|^{1+2s}} dr  d\mu, \; \; \int_{S_2^+\cap S^-_1} \int_{\{|r|\le 1\}} \dfrac{-u(x)}{|r|^{1+2s}} dr  d\mu, \; \;\mbox{and} \; \; \int_{S_2^+\cap S^-_2} \int_{\{|r|\le 1\}} \dfrac{-u(x)}{|r|^{1+2s}} dr  d\mu.$$
	Since all three integrals can be handled in the same way---using the fact that $|r|>d(x)$ whenever $x+\theta r \notin \Omega$ or $x-\theta r \notin \Omega$, with $\theta \in S^{n-1}$, where $d(x) \coloneqq \mbox{dist}(x,\partial\Omega)$---it suffices to establish the estimate for the first integral.
	
	Indeed, for $u \in W_0^{1,p}(\Omega)$
	\begin{eqnarray}
		\left(\int_{\Omega} \left[\int_{S_1^+\cap S^-_2} \int_{\{|r|\le 1\}} \dfrac{|u(x)|}{|r|^{1+2s}}\ dr \ d\mu\right]^p dx \right)^{1/p} &\le& \left(\int_{\Omega} \left[\int_{S_1^+\cap S^-_2} \int_{\{|r|\le 1\}} \dfrac{|u(x)|}{|r|^{2s}d(x)}\ dr \ d\mu\right]^p dx\right)^{1/p} \nonumber \\
		&\le& \left(\int_{\Omega} \left| \dfrac{u(x)}{d(x)}\right|^p dx\right) \left(\int_{S^{n-1}} \int_{\{|r|\le 1\}} \dfrac{1}{r^{2s}} \ dr \ d\mu\right)\nonumber \\
		&\le& C(p,\Lambda_1,s) \left(\int_{\Omega} \int_{0}^{1} |\nabla u(x+t(x-x_0))|^p \ dt \ dx\right)\nonumber \\
		&\le& C(p,\Lambda_1,s) \|u\|_{1,p}.\nonumber
	\end{eqnarray}
	where $x_0 \in \partial \Omega$ such that $d(x) = \mbox{dist}(x,x_0)$. Hence, by the previous estimates and \eqref{A} we conclude that
		\begin{equation}\label{estII}
			\|II\|_p \le C(p,\Lambda_1,s,\Omega) \|u\|_{1,p}.
		\end{equation}	
	Therefore, by combining \eqref{AB}, \eqref{estI}, \eqref{estII}, Sobolev interpolation and Young inequality we get for any $\varepsilon>0$
		$$\|L u(x)\|_p \le C \Big(\varepsilon \|u\|_{2,p} + \tau(\varepsilon) \|u\|_p\Big), \quad \mbox{for all} \ s\in \left(0,\frac{1}{2}\right).$$
		
	Now, we need to prove for the limiting case $s=\frac{1}{2}$. We observe that, when dividing the operator $L$ as in \eqref{AB}, the estimate of part $I$ follows exactly as before since $\int_{\{|r|> 1\}}|r|^{-2} dr$ remains finite. The same does not occur for the estimate of $II$, since $\int_{\{|r|\le1\}} |r|^{-2} dr$ diverges. Thus, we only need to improve the estimate of $II$.
	
	Suppose first that $p>n$. Then, the Sobolev embedding yields $u \in C^1(\overline{\Omega})$. Next, given $\varepsilon>0$, by similar arguments in \eqref{A}, we have by H\"{o}lder inequality, Fubini theorem and properties of the Extension operator that
		\begin{eqnarray*}
			&&\left(\int_\Omega \left[ \int_{S^+_1\cap S_1^-} \int_{\{|r|\le \varepsilon\}} \dfrac{u(x+\theta r) + u(x-\theta r)-2u(x)}{|r|^{2}} dr  d\mu\right]^p dx\right)^{1/p} \\
			&&\le \left(\int_\Omega \left[ \int_{S^+_1\cap S_1^-} \int_{\{|r|\le \varepsilon\}} \dfrac{\int_{0}^{1}\int_{-t}^{t}D_{ij}u(x+t'\theta r)\theta_i\theta_j r^2 dt' dt}{|r|^{2}} dr  d\mu\right]^p dx\right)^{1/p} \\
			&&\le \left(\int_\Omega \left[ \int_{S^+_1\cap S_1^-} \int_{\{|r|\le \varepsilon\}} \dfrac{\int_{0}^{1}\int_{-t}^{t}|D^2 u(x+t'\theta r)dt' dt}{|\theta|^{-2}} dr  d\mu\right]^p dx\right)^{1/p} \\
			&&\le \|Eu\|_{2,p} \left(\int_{S^{n-1}}\int_{\{|r|\le \varepsilon\}} |\theta|^2 \ dr \ d\mu\right) \\
			&&\le C(p,n,\Lambda_1,\Omega) o(\varepsilon)\|u\|_{2,p},
		\end{eqnarray*}
	while
		\begin{eqnarray}\label{I>ve}
			&&\left(\int_\Omega \left[ \int_{S^+_1\cap S^-_1} \int_{\{\varepsilon <|r|\le 1\}} \dfrac{u(x\pm\theta r) -u(x)}{|r|^{2}} dr  d\mu\right]^p dx\right)^{1/p} \nonumber \\
			&& \le \left(\int_{\mathbb{R}^n} \left[ \int_{S^{n-1}} \int_{\{\varepsilon<|r|\le 1\}} \int_{0}^{1}\dfrac{ |\nabla Eu(x\pm t\theta r)|}{|r|}dt dr  d\mu\right]^p dx\right)^{1/p}\nonumber \\
			&& \le \|Eu\|_{1,p}\left(\int_{S^{n-1}}\int_{\{\varepsilon<|r|\le 1\}} \dfrac{1}{|r|} \ dr \ d\mu\right)\nonumber\\
			&&\le C(p,n,\Lambda_1,\Omega) \tau(\varepsilon)\|u\|_{1,p},
		\end{eqnarray}
	where $o(\varepsilon)$ goes to $0$ and $\tau(\varepsilon)$ is arbitrary and unbounded, both as $\varepsilon \to 0$. Gathering the above estimates, we obtain that
		\begin{equation}
			\left(\int_\Omega \left[ \int_{S^+_1\cap S_1^-} \int_{\{|r|\le 1\}} \dfrac{u(x+\theta r) + u(x-\theta r)-2u(x)}{|r|^{2}} dr  d\mu\right]^p dx\right)^{1/p} \le C \Big(o(\varepsilon)\|u\|_{2,p} + \tau(\varepsilon)\|u\|_{1,p}\Big).
		\end{equation}
	The estimates of the integrals $B,C$ and $D$ follow the same line and here we will only prove for $B$. We once again define $d(x)\coloneqq \mbox{dist}\{x,\partial\Omega\}$ for any $x \in \Omega$. Since $u \in C^1(\overline{\Omega})$, with $u|_{\partial \Omega}=0$, and given that for any $p>n$, there exists $\gamma>0$ such that $p \le \frac{1}{\gamma}$, it follows that for any $\varepsilon \in (0,1)$ we have
		$$\int_{S_1^+\cap S_2^-} \int_{\{|r|\le \varepsilon\}} \dfrac{u(x+\theta r) -u(x)}{|r|^2} \ dr \ d\mu \le \int_{S_1^+\cap S_2^-} \int_{\{|r|\le \varepsilon\}} \dfrac{\|\nabla u(x)\|_\infty}{|r|^{1-\gamma}d(x)^{\gamma}}\ dr \ d\mu,$$
	as well as,
		$$\int_{S_1^+\cap S_2^-} \int_{\{|r|\le \varepsilon\}} \dfrac{-u(x)}{|r|^2} \ dr \ d\mu \le \int_{S_1^+\cap S_2^-} \int_{\{|r|\le \varepsilon\}} \dfrac{\|\nabla u(x)\|_\infty}{|r|^{1-\gamma}d(x)^{\gamma}}\ dr \ d\mu.$$
	Thus, by combining these estimates with the Sobolev inequality we obtain
		\begin{eqnarray*}
			&&\left(\int_\Omega\left[\int_{S^+_1\cap S_2^-} \int_{\{|r|\le \varepsilon\}} \dfrac{u(x+\theta r)-2u(x)}{|r|^{2}} dr  d\mu\right]^p dx \right)^{1/p} \\
			&& \le \left(\int_\Omega\left[\int_{S_1^+\cap S_2^-} \int_{\{|r|\le \varepsilon\}} \dfrac{\|\nabla u(x)\|_\infty}{|r|^{1-\gamma}d(x)^{\gamma}}\ dr \ d\mu\right]^p dx \right)^{1/p}\\
			&& \le C(n,p,\Lambda_1)o(\varepsilon)\|u\|_{2,p},
		\end{eqnarray*}
	which combined with \eqref{I>ve} produces
		$$\|B\|_p \le C\Big(o(\varepsilon)\|u\|_{2,p} + \tau(\varepsilon)\|u\|_{1,p}\Big).$$
	By same reasoning we may prove the estimates of $C$ and $D$. As a consequence of all these estimates, we obtain the desired inequality.
	
	The case $p=n$ allows us to claim that $u \in C^{0,1}(\overline{\Omega})$ which combined with the Sobolev inequality implies that $\|\nabla u(x)\|_\infty \le C(n,p)\|u\|_{2,n}$. Thus, the inequality follows from the case $p>n$.
	
	Finally, let us consider the case $p<n$. By same reasoning as before in the case $p>n$, we only need to verify the estimate of $B$ restrict to $\{|r|\le \varepsilon\}$. First note that
		\begin{eqnarray*}
			&&\left(\int_\Omega\left[\int_{S^+_1\cap S_2^-} \int_{\{|r|\le \varepsilon\}} \dfrac{u(x+\theta r)-u(x)}{|r|^{2}} dr  d\mu\right]^p dx \right)^{1/p}\nonumber \\
			&& \le \left(\int_\Omega\left[\int_{S_1^+\cap S_2^-} \int_{\{|r|\le \varepsilon\}} \dfrac{\int_0^1 |\nabla u(x+t\theta r)|dt}{|r|^{1-\gamma}d(x)^\gamma}\ dr \ d\mu\right]^p dx \right)^{1/p}\nonumber\\
			&& \le \left(\int_{\mathbb{R}^n}\left[\int_{S^{n-1}} \int_{\{|r|\le \varepsilon\}} \dfrac{\int_0^1 |\nabla Eu(x+t\theta r)|dt}{|r|^{1-\gamma}[d(x)^\gamma\chi_\Omega + (d(x)+1)\chi_{\mathbb{R}^n\setminus\Omega}]}\ dr \ d\mu\right]^p dx \right)^{1/p}\nonumber\\
			&&\le \left\|\dfrac{|\nabla Eu|}{[d(x)^\gamma\chi_\Omega + (d(x)+1)\chi_{\mathbb{R}^n\setminus\Omega}]}\right\|_p\left(\int_{S^{n-1}}\int_{\{|r|\le \varepsilon\}} \dfrac{1}{|r|^{1-\gamma}}\ dr \ d\mu\right)
		\end{eqnarray*}	
	Next, considering $q$ such that $p<n<q<\frac{1}{\gamma}$, the H\"{o}lder and Sobolev inequalities ensure that
		\begin{eqnarray}\label{hardine}
			&&\left(\int_\Omega\left[\int_{S^+_1\cap S_2^-} \int_{\{|r|\le \varepsilon\}} \dfrac{u(x+\theta r)-u(x)}{|r|^{2}} dr  d\mu\right]^p dx \right)^{1/p} \\
			&&\le C(p,\Lambda_1)\varepsilon^\gamma \left(\int_{\mathbb{R}^n} |\nabla Eu|^{\frac{pq}{q-p}}\ dx\right)^{\frac{q-p}{pq}}\left(\int_{\mathbb{R}^n} \dfrac{1}{[d(x)^\gamma\chi_\Omega + (d(x)+1)\chi_{\mathbb{R}^n\setminus\Omega}]^q} \ dx\right)^{1/q}\nonumber\\
			&&\le C(n,p,\Lambda_1,\Omega) \varepsilon^\gamma \left(\int_{\mathbb{R}^n} |\nabla Eu|^{\frac{pq}{q-p}}\ dx\right)^{\frac{q-p}{pq}}\\
			&&\le C(n,p,\Lambda_1,\Omega) \varepsilon^\gamma \|E u\|_{2,p} \nonumber \\
			&& \le C(n,p,\Lambda_1)o(\varepsilon)\|u\|_{2,p}, \nonumber
		\end{eqnarray}	
	where the inequality \eqref{hardine} is possible since $q \in (n,n/\alpha)$. Also, we get
		\begin{eqnarray*}
			\left(\int_\Omega \left[\int_{S^+_1\cap S_2^-} \int_{\{|r|\le \varepsilon\}} \dfrac{-u(x)}{|r|^2} \ dr \ d\mu\right]^p dx\right)^{\frac{1}{p}} &\le& \left(\int_\Omega \left[\int_{S^+_1\cap S_2^-} \int_{\{|r|\le \varepsilon\}} \dfrac{|u(x)|}{|r|^{1-\gamma}d(x)^{1+\gamma}} \ dr \ d\mu\right]^p dx\right)^{\frac{1}{p}} \\
			&\le& C(p,\Lambda_1,n,\Omega)\varepsilon^\gamma \left(\int_\Omega \left|\dfrac{u(x)}{d(x)}\right|^{\frac{pq}{q-p}} dx\right)^{\frac{q-p}{pq}} \\
			&\le& C(p,\Lambda_1,n,\Omega)\varepsilon^\gamma \left(\int_{\mathbb{R}^n} |\nabla Eu|^{\frac{pq}{q-p}} dx\right)^{\frac{q-p}{pq}} \\
			&\le& C(p,\Lambda_1,n,\Omega) o(\varepsilon) \|u\|_{2,p}.
		\end{eqnarray*}
	By combining the above two estimates, we obtain
		$$\|B|_{\{|r|<\varepsilon\}}\|_p \le C(p,\Lambda_1,n,\Omega) o(\varepsilon) \|u\|_{2,p}.$$
	Therefore, the result follows for all the cases.
\end{proof}

Now, the next lemma extend the previous result for the case $s \in (\frac{1}{2},1)$.

\begin{lemma}\label{L>12Lp}
	Let $\Omega$ be a $C^{1,1}$ domain in $\mathbb{R}^n$ and $s \in (\frac{1}{2},1)$. If $L$ is any operator of the form \eqref{defL} and \eqref{condL}, then for any $n<p<\frac{n}{2s-1}$ and $u \in W^{2,p}(\Omega) \cap W_0^{1,p}(\Omega)$, the operator $L$ satisfies
		$$\|Lu\|_p \le C \Big(o(\varepsilon) \|u\|_{2,p} + \tau(\varepsilon)\|u\|_p\Big), \quad \mbox{for any} \ \varepsilon>0,$$
	where the constant $C>0$ depends only on $n,p$, and $\Omega$, and $o(\varepsilon) \to 0$ and $\tau(\varepsilon)$ is unbounded, both as $\varepsilon\to 0$.
\end{lemma}
\begin{proof}
	Adopting the same notation as in the proof of Lemma \ref{L<12Lp}, it is sufficient to verify the estimate for $B$ restrict to the range $\{|r|\le \varepsilon\}$.
	
	To this end, we introduce $\Omega_\varepsilon \coloneqq \{x \in \Omega \; : \; d(x)>\varepsilon\}$ which excludes a thin boundary layer of width $\varepsilon$. Moreover, we fix a parameter $\gamma>0$ satisfying
		$$n<p<\dfrac{n}{\gamma} < \dfrac{n}{2s-1},$$
	a choice that will be crucial for balancing the integrability conditions in the subsequent estimates. Note that for any $x \in \Omega_\varepsilon$ the set $S_1^+\cap S_2^- \cap \{|r|\le \varepsilon\}$ is empty. Consequently,
		$$\left(\int_{\Omega_\varepsilon} \left[\int_{S^+_1\cap S_2^-}\int_{\{|r|\le \varepsilon\}} \dfrac{u(x+\theta r)-u(x)}{|r|^{1+2s}}\ dr \ d\mu\right]^p dx\right)^{1/p} =0,$$
	and likewise,
		$$\left(\int_{\Omega_\varepsilon} \left[\int_{S^+_1\cap S_2^-}\int_{\{|r|\le \varepsilon\}} \dfrac{-u(x)}{|r|^{1+2s}}\ dr \ d\mu\right]^p dx\right)^{1/p} =0.$$
	Hence, we obtain
		\begin{eqnarray*}
			&&\left(\int_{\Omega} \left[\int_{S^+_1\cap S_2^-}\int_{\{|r|\le \varepsilon\}} \dfrac{u(x+\theta r)-u(x)}{|r|^{1+2s}}\ dr\ d\mu\right]^p dx\right)^{1/p} \\
			&&\le \left(\int_{\Omega\setminus\Omega_\varepsilon} \left[\int_{S^+_1\cap S_2^-}\int_{\{|r|\le \varepsilon\}} \dfrac{\|\nabla u\|_\infty}{|r|^{2s-\gamma}d(x)^\gamma}\ dr\  d\mu\right]^p dx\right)^{1/p} \\
			&& \le C(p,s,\gamma,\Lambda_1) \|\nabla u\|_\infty \varepsilon^{\gamma-2s+1}\left(\int_{\Omega\setminus\Omega_\varepsilon} d(x)^{-\gamma p} \ dx\right)^{1/p} \\
			&& \le C(p,s,\gamma,\Lambda_1) \|\nabla u\|_\infty \varepsilon^{\gamma-2s+1}\left(\int_{y\in\partial\Omega}\int_{0<|x-y|\le \varepsilon} |x-y|^{-\gamma p} \ dx\ dy\right)^{1/p} \\
			&&\le C(p,s,\gamma,\Lambda_1) \|\nabla u\|_\infty \varepsilon^{\gamma-2s+1} \cdot\varepsilon^{-\gamma \frac{n}{p}}\\
			&&\le C(p,s,\gamma,\Lambda_1,\Omega)  \varepsilon^{\frac{n}{p} -2s+1}\|u\|_{2,p}.
		\end{eqnarray*}
	Similarly,
		\begin{eqnarray*}
			&&\left(\int_{\Omega} \left[\int_{S^+_1\cap S_2^-}\int_{\{|r|\le \varepsilon\}} \dfrac{-u(x)}{|r|^{1+2s}}\ dr\ d\mu\right]^p dx\right)^{1/p} \\
			&&\le \left(\int_{\Omega\setminus\Omega_\varepsilon} \left[\int_{S^+_1\cap S_2^-}\int_{\{|r|\le \varepsilon\}} \dfrac{\|\nabla u\|_\infty}{|r|^{2s-\gamma}d(x)^\gamma}\ dr\  d\mu\right]^p dx\right)^{1/p} \\
			&&\le C(p,s,\gamma,\Lambda_1,\Omega)  \varepsilon^{\frac{n}{p} -2s+1}\|u\|_{2,p}.
		\end{eqnarray*}
	These estimates ensure that
		$$\|B|_{\{|r|<\varepsilon\}}\|_p \le C(p,s,\gamma,\Lambda_1,\Omega) o(\varepsilon) \|u\|_{2,p}.$$
	Therefore, following the same reasoning as in the proof of Lemma \ref{L<12Lp} for the terms $I$, $A$, $B_{\{\varepsilon<|r|\le 1\}}$, $C$ and $D$, we conclude the desired result.
\end{proof}

\medskip\noindent	
\begin{tabular}{|l|}
	\hline
	Step two\\
	\hline
\end{tabular}

The main objective of this step is to establish, via the fixed point theorem, the existence and uniqueness of solutions to the linear perturbation
	$$-\Delta u + \lambda u = f -Lu.$$
We begin by recalling that if $p>1$ is given, there exists $\lambda_0$ such that the problem
	$$-\Delta u + \lambda u = f,$$
admits a unique solution $u \in W^{2,p}(\Omega)\cap W_0^{1,p}(\Omega)$ for every $\lambda\ge \lambda_0$ and $f\in L^p(\Omega)$. Moreover, the following estimates hold:
	\begin{equation}\label{estnorm2p}
		\|u\|_{2,p} \le C\|f\|_p,
	\end{equation}
and
	\begin{equation}\label{estnormp}
		(\lambda-\lambda_0)\|u\|_p \le C\|f\|_p,
	\end{equation}
where the constant $C>0$ depends only on the data of the problem, but is independent of $u$ and $\lambda$.

As before in the Step one, we divide the proof into two lemmas according to the range of $s$.

\begin{lemma}\label{existence<1/2}
	Let $\Omega\subset\mathbb{R}^n$ be a $C^{1,1}$-domain, and let $s \in (0,\frac{1}{2}]$, $f\in L^p(\Omega)$, with $p>1$. Assume that $L$ is any operator of the form \eqref{defL} and \eqref{condL}. Then, there exists $\widetilde{\lambda}>0$ sufficiently large and independent of $f$, such that the problem
		$$Lu -\Delta u + \lambda u = f, \quad \mbox{in} \ \Omega,$$
	admits a unique solution $u \in W^{2,p}(\Omega)\cap W_0^{1,p}(\Omega)$ for every $\lambda\ge \widetilde{\lambda}$.
\end{lemma}
\begin{proof}
	For any $w \in W^{2,p}(\Omega)\cap W_0^{1,p}(\Omega)$, Lemma \ref{L<12Lp} guarantees that $f - Lw \in L^p(\Omega)$. Therefore, the problem
	\begin{equation}\label{solw}
		-\Delta u + \lambda u = f - Lw, \quad \text{in } \Omega,
	\end{equation}
	admits a unique solution $u \in W^{2,p}(\Omega)\cap W_0^{1,p}(\Omega)$. This allows us to define the operator
	$$
	T_\lambda : W^{2,p}(\Omega)\cap W_0^{1,p}(\Omega) \longrightarrow W^{2,p}(\Omega)\cap W_0^{1,p}(\Omega),
	$$
	by setting $T_\lambda w = u$, where $u$ is the unique solution of \eqref{solw}.
	
	The strategy is then to show that $T_\lambda$ is a contraction. Once this is established, the Banach fixed point theorem yields the existence and uniqueness of the desired solution.
	
	Indeed, let $w_1,w_2 \in W^{2,p}(\Omega)\cap W_0^{1,p}(\Omega)$ and $T_\lambda w_1 =u_1$ and $T_\lambda w_2 =u_2$. Then,
		$$-\Delta\big(u_1-u_2\big) + \lambda(u_1-u_2) = -L\big(w_1-w_2\big), \quad \mbox{in}\ \Omega.$$
	Let $\delta \in (0,1)$ be chosen later. By \eqref{estnorm2p} and Lemma \ref{L<12Lp}, we obtain
		\begin{eqnarray}\label{u1u2A}
			\|u_1-u_2\|_{2,p} &\le& C \|L(w_1-w_2)\|_p \nonumber\\
			&\le& C\left(\big(o(\varepsilon) + \tau(\varepsilon)\delta\big)\|w_1-w_2\|_{2,p} + \dfrac{\tau(\varepsilon)}{\delta}\|w_1-w_2\|_p\right).
		\end{eqnarray}
	By the same reasoning, but using \eqref{estnormp} instead of \eqref{estnorm2p}, we deduce
		\begin{equation}\label{u1u2B}
			(\lambda-\lambda_0)\|u_1-u_2\|_p \le C\left(\big(o(\varepsilon) + \tau(\varepsilon)\delta\big)\|w_1-w_2\|_{2,p} + \dfrac{\tau(\varepsilon)}{\delta}\|w_1-w_2\|_p\right).
		\end{equation}
	Since the constant $C$ is independent of $\lambda$, we may introduce the norm (equivalent to the standard one) in $W^{2,p}(\Omega)\cap W_0^{1,p}(\Omega)$ defined by
		$$\|\cdot\|_{\widetilde{2,p}}\coloneqq \|\cdot\|_{2,p} + (\lambda-\lambda_0)\|\cdot\|_p.$$
	Let $\gamma \in (0,1)$ be fixed. We first choose $\varepsilon > 0$ sufficiently small so that
		$$C\big(o(\varepsilon) + \tau(\varepsilon)\delta\big) \le \gamma.$$
	Then we take $\widetilde{\lambda} > 0$ large enough such that
		$$C\dfrac{\tau(\varepsilon)}{\delta} \le (\widetilde{\lambda} -\lambda_0)\gamma.$$
	It follows from \eqref{u1u2A} and \eqref{u1u2B} that for all $\lambda \geq \widetilde{\lambda}$,
		\begin{eqnarray*}
			\|T_\lambda w_1 - T_\lambda w_2\|_{\widetilde{2,p}} &=& \|u_1-u_2\|_{2,p} + (\lambda -\lambda_0)\|u_1-u_2\|_p \\
			&\le& C\left(\big(o(\varepsilon) + \tau(\varepsilon)\delta\big)\|w_1-w_2\|_{2,p} + \dfrac{\tau(\varepsilon)}{\delta}\|w_1-w_2\|_p\right)\\
			&\le& \gamma \Big(\|w_1-w_2\|_{2,p} + (\lambda-\lambda_0)\|w_1-w_2\|_p\Big) \\
			&=&\gamma \|w_1-w_2\|_{\widetilde{2,p}}.
		\end{eqnarray*}
	Therefore, $T_\lambda$ is a contraction, and the result follows by the Banach fixed point theorem.
\end{proof}

For the case $s \in (\frac{1}{2},1)$ we have the following:
\begin{lemma}\label{existence>1/2}
	Let $\Omega\subset\mathbb{R}^n$ be a $C^{1,1}$ domain, and let $s \in (\frac{1}{2},1)$, $f\in L^p(\Omega)$, with $n<p<\frac{n}{2s-1}$. Assume that $L$ is any operator of the form \eqref{defL} and \eqref{condL}. Then, there exists $\widehat{\lambda}>0$ sufficiently large and independent of $f$, such that the problem
	$$Lu -\Delta u + \lambda u = f, \quad \mbox{in} \ \Omega,$$
	admits a unique solution $u \in W^{2,p}(\Omega)\cap W_0^{1,p}(\Omega)$ for every $\lambda\ge \widehat{\lambda}$.
\end{lemma}
\begin{proof}
	The proof follows {\it ipsis litteris} the argument used in Lemma \ref{existence<1/2}, with the only difference that we apply Lemma \ref{L>12Lp} in place of Lemma \ref{L<12Lp}.
\end{proof}

\medskip\noindent	
\begin{tabular}{|l|}
	\hline
	Step three\\
	\hline
\end{tabular}

In this final step, we turn to the proof of Theorem \ref{W2p-reg}. As a preliminary step, we first establish a few auxiliary results concerning the Maximum Principle.

\begin{lemma}\label{linftyestimate}
	Let $\Omega \subset \mathbb{R}^n$ be a $C^{1,1}$-domain, $s \in (0,1)$, and $L$ be any operator of the form \eqref{defL} and \eqref{condL}. Assume that $g \in L^\infty(\Omega)$ and $\lambda \ge \max \{\widetilde{\lambda},\widehat{\lambda}\}$, where $\widetilde{\lambda}$ and $\widehat{\lambda}$ are given in Lemmas \ref{existence<1/2} and \ref{existence>1/2}, respectively. Then, the problem
		$$Lu - \Delta u + \lambda u = g, \quad \mbox{in}\ \Omega,$$
	has a unique solution $u\in W^{2,p}(\Omega)\cap W_0^{1,p}(\Omega)$, with $p$ satisfying \eqref{rangep}, and
		$$\|u\|_\infty \le C_\lambda \|g\|_\infty,$$
	where $C_\lambda <\dfrac{2}{\lambda}$.
\end{lemma}

We observe that, since $\Omega$ is a bounded domain, existence and uniqueness in the previous lemma follow directly from Lemma \ref{existence<1/2} or Lemma \ref{existence>1/2}, depending on the range of $s$. Hence, it remains to establish the $L^\infty$-norm estimate. The proof of this estimate will be based on the following two lemmas.

\begin{lemma}\label{vcinfty}
	Let $L$ be any operator of the form \eqref{defL} and \eqref{condL}, with $s \in (0,1)$. There exists $w\in C^\infty(\overline{\Omega})\cap W^{1,\infty}(\mathbb{R}^n)$ such that
		$$\left\{
		\begin{array}{rccl}
			Lw - \Delta w&\le& 1, & \mbox{in} \ \Omega; \\
			w&>& 0, & \mbox{in}\ \overline{\Omega};\\
			w&\ge& 0, & \mbox{in} \ \mathbb{R}^n. 	
		\end{array}
		\right.$$
\end{lemma}
\begin{proof}
	Since $\Omega$ is a bounded domain in $\mathbb{R}^n$, there exist a point $x_0 \in \mathbb{R}^n$ and a constant $R>0$ such that $\overline{\Omega} \subset B_{\frac{3}{4}R}(x_0) \setminus B_{\frac{1}{4}R}(x_0)$. Without loss of generality, we may assume $x_0=0$. For convenience, set $B_R = B_R(0)$. Now, define the auxiliary function
		$$w(x) \coloneqq\left\{
		\begin{array}{rl}
			1-e^{\beta(|x|^2-R^2)}, &\mbox{if}\ |x|\le R; \\
			0, &\mbox{if} \ |x|>R,
		\end{array}
		\right.$$
	where the parameter $\beta > 0$ will be fixed later. We can see, that $w \in C^\infty(\overline{\Omega})\cap W^{1,\infty}(\mathbb{R}^n)$, as well as, $w> 0$ in $\Omega$ and $w\ge0$ in $\mathbb{R}^n$. Further, by direct computations we know that
		$$-\Delta w(x) = e^{\beta(|x|^2-R^2)}\big(2n\beta + 4\beta^2|x|^2\big), \quad \mbox{for any} \ x \in \Omega.$$
	Now, we turn to estimating the fractional part, namely the quantity $Lw$. First, observe that if $x \in \Omega$ and $|r|\le \tfrac{R}{4}$, then for every $\theta \in S^{n-1}$ we have $x \pm \theta r \in B_R$. Then, since $w$ is concave on $\overline{B_R}$, for any $x \in \Omega$, $\theta \in S^{n-1}$, and $|r|\le \frac{R}{4}$ we have $w(x+\theta r)+w(x-\theta r) \le 2w(x)$. This implies that the contribution to $Lw$ from the region $\{|r|\le \tfrac{R}{4}\}$ is non-positive. Consequently, for any $x \in \Omega$ we obtain
		\begin{eqnarray*}
			Lw(x) &=& \int_{S^{n-1}}\int_{-\infty}^\infty \dfrac{w(x+\theta r) + w(x-\theta r)-2w(x)}{|r|^{1+2s}}\ dr\ d\mu \\
			&\le& \int_{S^{n-1}}\int_{\{|r|>\frac{R}{4}\}} \dfrac{w(x+\theta r) + w(x-\theta r)-2w(x)}{|r|^{1+2s}}\ dr\ d\mu \\
			&=& \int_{S^{n-1}}\int_{\{|r|>\frac{R}{4}\}} \dfrac{w(x+\theta r) -w(x)}{|r|^{1+2s}}\ dr\ d\mu + \int_{S^{n-1}}\int_{\{|r|>\frac{R}{4}\}} \dfrac{w(x-\theta r) -w(x)}{|r|^{1+2s}}\ dr\ d\mu \\
			&\eqcolon& A + B.
		\end{eqnarray*}
		
	Next, we rewrite the first term $A$ by splitting the integration domain as follows
		\begin{eqnarray*}
			A &=& \int_{S^{n-1}}\int_{\{|r|>\frac{R}{4}\}\cap\{|x+\theta r|\le |x|\}} \dfrac{w(x+\theta r) -w(x)}{|r|^{1+2s}}\ dr\ d\mu \\
			&& + \int_{S^{n-1}}\int_{\{|r|>\frac{R}{4}\}\cap\{|x+\theta r|> |x|\}} \dfrac{w(x+\theta r) -w(x)}{|r|^{1+2s}}\ dr\ d\mu.
		\end{eqnarray*}
	For the first integral, it suffices to observe that, by the very definition of $w$, we have
		\begin{eqnarray*}
			\int_{S^{n-1}}\int_{\{|r|>\frac{R}{4}\}\cap\{|x+\theta r|\le |x|\}} \dfrac{w(x+\theta r) -w(x)}{|r|^{1+2s}}\ dr\ d\mu &\le& \int_{S^{n-1}}\int_{\{|r|>\frac{R}{4}\}} \dfrac{e^{\beta(|x|^2-R^2)}}{|r|^{1+2s}}\ dr\ d\mu \\
			&\le& c(\Lambda_1,s,R) e^{\beta(|x|^2-R^2)}.
		\end{eqnarray*}
	For the second integral, we distinguish two different situations depending on the position of the modulus of $x+\theta r$.
		\begin{itemize}
			\item {\bf Case 1:} $|x|\le |x+\theta r|\le R$.
			
				In this situation, both $x$ and $x+\theta r$ lie inside $B_R$, which implies, by the monotonicity of $e^{(\cdot)}$, that
					$$w(x+\theta r) - w(x) = e^{\beta(|x|^2 -R^2)} - e^{\beta(|x+\theta r|^2 -R^2)}  \le 0.$$
					
			\item {\bf Case 2:} $|x|\le R < |x+\theta r|$.
			
				Here, by the definition of $w$ we compute
					$$w(x+\theta r)-w(x) = e^{\beta(|x|^2-R^2)} -1  \le e^{\beta(|x|^2-R^2)}.$$
		\end{itemize}		
	Collecting the two cases, we deduce that
		\begin{eqnarray*}
			\int_{S^{n-1}}\int_{\{|r|>\frac{R}{4}\}\cap\{|x+\theta r|> |x|\}} \dfrac{w(x+\theta r) -w(x)}{|r|^{1+2s}}\ dr\ d\mu &\le& \int_{S^{n-1}} \int_{\{|r|>\frac{R}{4}\}} \dfrac{ e^{\beta(|x|^2-R^2)}}{|r|^{1+2s}}\ dr \ d\mu \\
			&\le& c(\Lambda_1,s,R)  e^{\beta(|x|^2-R^2)}.
		\end{eqnarray*}
	Therefore, combining with the estimate for the first integral, we obtain
		$$A \le c(\Lambda_1,s,R) e^{\beta(|x|^2-R^2)}.$$
	By a completely analogous reasoning, one can show that $B \le c(\Lambda_1,s,R) e^{\beta(|x|^2-R^2)}$. Consequently, for any $x \in \Omega$ we have
		$$L w(x) \le c(\Lambda_1,s,R) e^{\beta(|x|^2-R^2)}, \quad \mbox{for any} \ x \in \Omega.$$
		
	Finally, since in our region of interest $\Omega$ we have $\tfrac{1}{4}R \le |x| \le \tfrac{3}{4}R$, it follows that
		\begin{eqnarray*}
			Lw(x) - \Delta w &\le& c(\Lambda_1,s,R) e^{\beta(|x|^2-R^2)} + e^{\beta(|x|^2-R^2)}\Big(2n\beta + 4\beta^2|x|^2\Big) \\
			&=& e^{\beta(|x|^2-R^2)}\Big(2n\beta + 4\beta^2|x|^2 + c(\Lambda_1,s,R) \Big) \\
			&\le&e^{-\frac{7\beta R^2}{16}}\left(2n\beta +\dfrac{9\beta^2R^2}{4} + c(\Lambda_1,s,R)\right).
		\end{eqnarray*}
	Thus, by choosing $\beta$ sufficiently large, we ensure that the right-hand side is less than or equal to $1$, which concludes the proof.
\end{proof}

\begin{lemma}\label{=1linftyestimate}
	Let $L$ be any operator of the form \eqref{defL} and \eqref{condL}, with $s \in (0,1)$. Then, the problem
		$$Lu - \Delta u + \lambda u =1, \quad \mbox{in} \ \Omega,$$
	has a unique solution $v_\lambda \in W^{2,p}(\Omega)\cap L^\infty(\Omega)$, with $p$ satisfying \eqref{rangep}, and $\|v_\lambda\|_\infty \le \frac{2}{\lambda}$.
\end{lemma}
\begin{proof}
	The existence and uniqueness of the solution, say $v_\lambda$, is guaranteed by Lemmas \ref{existence<1/2} and \ref{existence>1/2}, depending on the range of $s$. Furthermore, the Sobolev embedding theorem implies that $v_\lambda \in C^0(\overline{\Omega}) \cap L^\infty(\Omega)$. Also, by Lemma \ref{vcinfty}, it follows that there exists $w \in C^\infty(\overline{\Omega})\cap W^{1,\infty}(\mathbb{R}^n)$ with $w>0$ in $\Omega$,$w \ge 0$ in $\mathbb{R}^n$ and
		$$Lw-\Delta w \le 1 \quad \mbox{in} \ \Omega.$$
	Set $\varphi(t)=\frac{1}{\lambda}\big(1+e^{-\lambda t}\big)$, and observe that $\varphi$ is a convex function with maximum value $\max_{t\ge0}\varphi(t)=\frac{2}{\lambda}$. Hence, for any $x,y \in \mathbb{R}^n$ we have the convexity inequality
		$$\varphi(w(x)) - \varphi(w(y)) \ge \varphi'(w(y))\big(w(x)-w(y)\big).$$
	Let us compute $L\varphi(w(x))$ for $x \in \Omega$. Using the above inequality, we obtain
		\begin{eqnarray*}
			L\varphi(w(x)) &\ge& \int_{S^{n-1}}\int_{-\infty}^\infty \dfrac{\varphi'(w(x))\big(w(x+\theta r)-w(x)\big) + \varphi'(w(x))\big(w(x-\theta r)-w(x)\big)}{|r|^{1+2s}}\ dr \ d\mu \\
			&=& \varphi'(w(x))  \int_{S^{n-1}}\int_{-\infty}^\infty \dfrac{w(x+\theta r)+ w(x-\theta r)-2w(x)}{|r|^{1+2s}}\ dr \ d\mu \\
			&=& -e^{-\lambda w(x)} L w(x).
		\end{eqnarray*}
	As a consequence, since $Lw - \Delta w \le 1$, it follows that
		\begin{eqnarray}\label{ldlam}
			L\varphi (w(x)) - \Delta \varphi(w(x)) + \lambda \varphi(w(x)) &\ge& -e^{\lambda w(x)} \left(Lw(x) - \Delta w -1 - \lambda \sum_{i=1}^n \left(\dfrac{\partial w}{\partial x_i}\right)^2 \right) +1 \nonumber \\
			&\ge& 1.
		\end{eqnarray}
	On the other hand, recall that $v_\lambda$ satisfies $L v_\lambda - \Delta v_\lambda +\lambda v_\lambda =1$. Hence, comparing this with \eqref{ldlam}, we obtain that
		$$L\big(\varphi(w) - v_\lambda\big) - \Delta \big(\varphi(w) - v_\lambda\big) + \lambda \big(\varphi(w) - v_\lambda\big) \ge 0, \quad \mbox{in}\ \Omega.$$
	By the maximum principle (Theorem \ref{maxprin}), we get
		$$v_\lambda(x) \le \varphi(w(x)), \quad \mbox{a.e. in}\ \Omega.$$
	Finally, since $\max_{x\in\Omega} \varphi(w(x)) = \frac{2}{\lambda}$, we conclude that $\|v_\lambda\|_\infty \le \frac{2}{\lambda}$, which is the desired estimate.
\end{proof}

Finally, we are ready to prove Lemma \ref{linftyestimate}, as follows.
	
\begin{proof}[Proof of Lemma \ref{linftyestimate}]
	Observe that
		$$L\left(\dfrac{u}{\|g\|_\infty}\right) - \Delta \left(\dfrac{u}{\|g\|_\infty}\right) + \lambda \left(\dfrac{u}{\|g\|_\infty}\right) \le 1, \quad \mbox{in} \ \Omega.$$
	On the other hand, combining this inequality with Lemma \ref{=1linftyestimate}, we obtain
		$$L \left(v_\lambda - \dfrac{u}{\|g\|_\infty}\right) - \Delta \left(v_\lambda - \dfrac{u}{\|g\|_\infty}\right) + \lambda \left(v_\lambda - \dfrac{u}{\|g\|_\infty}\right) \ge 0, \quad \mbox{in} \ \Omega.$$
	By applying the Theorem \ref{maxprin} we get $\|u\|_\infty \le v_\lambda \|g\|_\infty$. Finally, the desired bound follows directly from the estimate of $\|v_\lambda\|_\infty$ provided in Lemma \ref{=1linftyestimate}.
\end{proof}	

Finally, we are able to prove the Theorem \ref{W2p-reg}, as follows.

\begin{proof}[Proof of Theorem \ref{W2p-reg}]
	Let us start by fixing $\lambda > \max\left\{\widetilde{\lambda},\widehat{\lambda}\right\}$, where the constants $\widetilde{\lambda}$ and $\widehat{\lambda}$ are given in Lemmas \ref{existence<1/2} and \ref{existence>1/2}, respectively.
	
	We define a sequence of functions $\{u_k\}$ as follows. Let $u_0$ be the unique solution of
		$$Lu_0 -\Delta u_0 + \lambda u_0 = f, \quad \mbox{in} \ \Omega,$$
	with $u_0 \in W^{2,p}(\Omega) \cap W_0^{1,p}(\Omega)$, provided by Lemmas \ref{existence<1/2} and \ref{existence>1/2}, according to the range of $s$. Since $\lambda u_0$ belongs to $L^p(\Omega)$, we may consider, as before, $u_1$ as the unique solution of
		$$Lu_1 -\Delta u_1 + \lambda u_1 = f + \lambda u_0, \quad \mbox{in} \ \Omega,$$
	with $u_1 \in W^{2,p}(\Omega) \cap W_0^{1,p}(\Omega)$. Proceeding inductively, for any $k \in \mathbb{N}$ we set $u_k$ as the unique solution of
		\begin{equation}\label{eqtopasslimit}
			Lu_{k} -\Delta u_{k} + \lambda u_{k} = f + \lambda u_{k-1}, \quad \mbox{in} \ \Omega,
		\end{equation}
	with $u_{k} \in W^{2,p}(\Omega) \cap W_0^{1,p}(\Omega)$, where we define.
	
	Next, by mimicking the proof of Theorem 4.1 in \cite{SVWZ}, we may ensure that
		$$\|u_{k+1}-u_k\|_\infty \le K \|u_{k}-u_{k-1}\|_\infty,$$
	for some constant $K \in (0,1)$. Consequently, the sequence $\{u_k\}$ is Cauchy in $L^\infty(\Omega)$ and, in particular, bounded.
	
	By Lemmas \ref{L<12Lp} and \ref{L>12Lp}, we know that $Lu_{k+1} \in L^p(\Omega)$ for any $s \in (0,1)$. Combining these lemmas with \eqref{estnorm2p} and \eqref{estnormp}, we obtain
		\begin{eqnarray*}
			\|u_{k+1}\|_{2,p} + (\lambda -\lambda_0)\|u_{k+1}\|_p &\le& 2\Big(\|f\|_p + \lambda \|u_k\|_p + \|Lu_{k+1}\|_p\Big) \\
			&\le&2C\Big(o(\varepsilon) \|u_{k+1}\|_{2,p} + \tau(\varepsilon)\|u_{k+1}\|_{1,p} + \|f\|_p + \lambda \|u_k\|_p\Big),
		\end{eqnarray*}
	where the constant $C>0$ is independent of both $\lambda$ and $u_k$. Now, fix any $\vartheta \in (0,1)$ independent of $k$. By choosing $\varepsilon>0$ sufficiently small, we may ensure that
		$$\|u_{k+1}\|_{2,p} + (\lambda -\lambda_0)\|u_{k+1}\|_p \le \vartheta\Big( \|u_{k+1}\|_{2,p} + (\lambda-\lambda_0) \|u_{k+1}\|_{p}\Big) + C\|f\|_p + C\lambda \|u_k\|_\infty,$$
	and then, it follows that the sequence $\{u_k\}$ is bounded in $W^{2,p}(\Omega) \cap W_0^{1,p}(\Omega)$. Thus, there exists a subsequence (not relabeled) converging weakly to some $u \in W^{2,p}(\Omega) \cap W_0^{1,p}(\Omega)$.
	
	Moreover, Lemmas \ref{L<12Lp} and \ref{L>12Lp} imply that $L$ is a linear bounded operator from $W^{2,p}(\Omega) \cap W_0^{1,p}(\Omega)$ into $L^p(\Omega)$. Therefore, for every $g \in L^{\frac{p}{p-1}}(\Omega)$,
		$$\int_\Omega gL u_k\ dx \to \int_\Omega gL u\ dx.$$
	In addition, standard weak convergence arguments ensure that
		$$\int_\Omega g(-\Delta u_k)\ dx \to \int_\Omega g(-\Delta u)\ dx,$$
	for all $g \in L^{\frac{p}{p-1}}(\Omega)$, as well. Passing to the limit of $k\to \infty$ in \eqref{eqtopasslimit}, the above results conclude that
		$$Lu -\Delta u = f \quad \mbox{in} \ \Omega.$$
	Finally, applying the Agmon-Douglis-Nirenberg estimate together with Lemmas \ref{L<12Lp} and \ref{L>12Lp}, and choosing $o(\varepsilon) \le \tfrac{1}{2C}$, we deduce
		$$\|u\|_{2,p} \le \mathrm{C}\Big(\|u\|_p + \|f\|_p\Big),$$
	as desired.
\end{proof}

Now we complete the proof of Theorem \ref{boundary_reg}.

\begin{proof}[Proof of Theorem \ref{boundary_reg}]
	Since $f \in L^\infty(\Omega)$, Theorem \ref{W2p-reg} yields that the problem
		$$Lu - \Delta u =f,$$
	has a unique solution $u \in W^{2,p}(\Omega)\cap W_0^{1,p}(\Omega)$ with $p$ satisfying \eqref{rangep}. Finally, by combining this with the Sobolev inequality, we obtain that $u \in C^{1,\gamma}(\overline{\Omega})$, with $\gamma>0$ satisfying
		$$\left\{
		\begin{array}{rc}
			\gamma \in (0,1), & \mbox{if} \ s\in (0,\frac{1}{2}]; \\
			\gamma \in (0,2-2s), & \mbox{if}\ s\in (\frac{1}{2},1),
		\end{array}
		\right.$$
	and the result follows.
\end{proof}

%
%


\begin{thebibliography}{99}
%
	
	\bibitem{BI} G. Barles; C. Imbert, \emph{Second-order elliptic integro-differential equations: viscosity solutions’ theory revisited}, Ann. Inst. H. Poincé C Anal. Non Linéaire, 25, 567–-585, (2008).
	
	\bibitem{B} R. Bass; \emph{Regularity results for stable-like operators}, J. Funct. Anal., 257, 2693-–2722, (2009).
	
	\bibitem{BC} R. Bass; Z. Chen, \emph{Regularity of harmonic functions for a class of singular stable-like processes}, Math. Z., 266, 489-–503, (2010).
	
	
	\bibitem{BDVV1} {S. Biagi; S. Dipierro; E. Valdinoci; E. Vecchi}, \emph{A {B}rezis-{N}irenberg type result for mixed local and
		nonlocal operators}, {NoDEA Nonlinear Differential Equations Appl.}, {32}, {4}, (2025).
		
	\bibitem{BDVV2} {S. Biagi; S. Dipierro; E. Valdinoci; E. Vecchi}, \emph{Mixed local and nonlocal elliptic operators: regularity and
		maximum principles}, {Comm. Partial Differential Equations}, {47}, {3}, (2022).
	
	\bibitem{BDVV} {S. Biagi; S. Dipierro; E. Valdinoci; E. Vecchi}, \emph{Semilinear elliptic equations involving mixed local and nonlocal operators}, Proc. R. Soc. Edinb., Sect. A, 151, (5), 1611–-1641, (2021).
	
	\bibitem{BDG} {I. Birindelli; L. Du; G. Galise}, \emph{Liouville results for semilinear integral equations with conical diffusion}, {Calc. Var. Partial Differential Equations}, {64}, {3}, {Paper No. 86, 26}, (2025).
	
	\bibitem{BS} K. Bogdan; P. Sztonyk, \emph{Harnack’s inequality for stable L\'{e}vy processes}, Potential Anal., 22, 133–-150, (2005).
	
	\bibitem{BNR} T. Byczkowski; J. Nolan; B. Rajput, A\emph{pproximation of multidimensional stable densities}, J. Multivariate Anal., 46, 13-–31, (1993).
	
	\bibitem{BLS} {S-S. Byun; H-S, Lee; K. Song}, \emph{Regularity results for mixed local and nonlocal double phase functionals}, {J. Differential Equations}, {416}, 1528--1563, (2025).
	
	\bibitem{CS} L. Caffarelli; L. Silvestre, \emph{Regularity theory for fully nonlinear integro-differential equations}, Comm. Pure Appl. Math., 62, 597-–638, (2009).
	
	\bibitem{CKSV1} {Z. Chen; P. Kim; R. Song, Z. Vondra\v cek}, \emph{Sharp Green function estimates for $\Delta+ \Delta^{\alpha/2}$ in $C^{1.1}$ open sets and their applications}, III. J. Math. 54, 981-–1024, (2010).
	
	\bibitem{CKS} {Z. Chen; P. Kim; R. Song}, \emph {Heat kernel estimates for {$\Delta+\Delta^{\alpha/2}$} in
		{$C^{1,1}$} open sets}, {J. Lond. Math. Soc. (2)}, {84}, {1}, {58--80}, (2011).
	
	\bibitem{CKSV2} {Z. Chen; P. Kim; R. Song, Z. Vondra\v cek}, \emph{Boundary Harnack principle for $\Delta+ \Delta^{\alpha/2}$}, Trans. Am. Math. Soc., 364, 4169--4205, (2012).
	
	\bibitem{CK} Z. Chen; T. Kumagai, \emph{A priori H\"{o}lder estimate, parabolic Harnack principle and heat kernel estimates for diffusions with jumps}, Rev. Mat. Iberoam., 26, 551–-589, (2010).
	
%
	
	\bibitem{DfM} C. De Filippis; G. Mingione, \emph{Gradient regularity in mixed local and nonlocal problems}, Math. Ann., 388, (1), 261–-328, (2024).
	
	\bibitem{DLV} {S. Dipierro, E.P. Lippi, E. Valdinoci}, \emph{(Non)local logistic equations with Neumann conditions}, Ann. Inst. Henri Poincaré, Anal. Non Linéaire, 40, (5),  1093--1166, (2023).
	
	\bibitem{DV} {S. Dipierro; E Valdinoci}, \emph{Description of an ecological niche for a mixed local/nonlocal
		dispersal: an evolution equation and a new {N}eumann condition
		arising from the superposition of {B}rownian and {L}\'evy
		processes}, {Phys. A}, {575}, {Paper No. 126052, 20}, (2021).
	
	\bibitem{DY} {L. Du; M. Yang}, \emph{On elliptic equations with $N$-independent stable operators}, https://doi.org/10.48550/arXiv.2501.00198
	
	
	
	\bibitem{FR} X. Fernández-Real; X. Ros-Oton, \emph{Regularity theory for general stable operators: Parabolic equations}, J. Funct. Anal., 272, 4165–-4221, (2017).
	
	\bibitem{F} M. Foondun, \emph{Heat kernel estimates and Harnack inequalities for some Dirichlet forms with non-local part}, Electron. J. Probab., 14, 314–-340, (2009).
	
	\bibitem{GK} {P. Garain, J. Kinnunen}, \emph{On the regularity theory for mixed local and nonlocal quasilinear elliptic equations}, {Trans.	Am. Math. Soc.}, {375}, (8), 5393-5423, (2022).
	
	
	\bibitem{GT} {D. Gilbarg; N.S. Trudinger}, \emph{Elliptic Partial Differential Equations of Second Order,} {Grundlehren der Mathematis-chen Wissenschaften,} {vol.224,} {Springer-Verlag,} {Berlin–New York,} (1977).
	
	
	
	
	
	\bibitem{MS} R. Mantegna; H. Stanley, \emph{Scaling behaviour in the dynamics of an economic index}, Nature, 376, 46-–49, (1995).
	
%
	
	\bibitem{N} J. Nolan, \emph{Fitting data and assessing goodness-of-fit with stable distributions}, Applications of Heavy Tailed Distributions in Economics, Engineering and Statistics, Washington, DC, (1999).
	
	
	\bibitem{RS}{X. Ros-Oton; J. Serra}, \emph{Regularity theory for general stable operators}, {J. Differential Equations}, {260}, {8675--8715}, (2016).
	
	\bibitem{ST} G. Samorodnitsky; M. Taqqu, \emph{Stable Non-Gaussian Random Processes: Stochastic Models with Infinite Variance}, Chapman and Hall, New York, (1994).
	
	\bibitem{S}{L. Simon,} \emph{Schauder estimates by scaling,} {Calc. Var. Partial Differential Equations} {5} {391--407}, (1997).
	
	\bibitem{SVWZ2} {X. Su; E. Valdinoci; Y. Wei; J. Zhang} \emph{On some regularity properties of mixed local and nonlocal elliptic equations}, {J. Differential Equations}, {416}, {576--613}, (2025).
	
	\bibitem{SVWZ} {X. Su; E. Valdinoci; Y. Wei; J Zhang}, \emph{Regularity results for solutions of mixed local and nonlocal elliptic equations}, {Math. Z.}, {302}, {3}, {1855--1878}, (2022).
	
	
\end{thebibliography}
\end{document}